\documentclass[11pt,letterpaper,reqno]{article}

\usepackage{textcomp}
\usepackage{amsmath,amsthm}
\usepackage{amsfonts}
\usepackage{amscd}
\usepackage[all]{xy}
\usepackage{amssymb}
\usepackage[a4paper]{geometry}
\usepackage{mathtools}
\usepackage{float}
\usepackage{subfig}
\usepackage{graphicx}
\usepackage[font=small,labelfont=bf]{caption}
\usepackage[english]{babel}

\usepackage{hyperref}

\usepackage{tikz}
\usetikzlibrary{matrix}
\usepackage{authblk}
\usepackage{lipsum}
\usepackage{tikz-cd}
\usepackage{mathrsfs}

\newtheorem{proposition}{Proposition}[section]
\newtheorem{lemma}{Lemma}[section]
\newtheorem{corollary}{Corollary}[section]
\newtheorem{definition}{Definition}[section]
\newtheorem{remark}{Remark}[section]
\newtheorem{question}{Question}
\newtheorem{theorem}{Theorem}[section]
\newtheorem*{theoremA}{Theorem A}
\newtheorem*{theoremB}{Theorem B}
\newtheorem*{theoremC}{Theorem C}
\newtheorem*{theoremD}{Theorem D}

\newcommand{\norm}[1]{\left\Vert#1\right\Vert}
\newcommand{\conorm}[1]{\mathrm{m}\left(#1\right)}
\newcommand{\Real}{\mathbb R}

\newcommand{\Tor}{\mathbb{T}^2}

\newcommand{\Pp}{\mathbb P}

\newcommand{\R}{\mathbb{R}}
\newcommand{\Z}{\mathbb{Z}}
\newcommand{\N}{\mathbb{N}}
\newcommand{\TT}{\mathbb{T}^2}
\newcommand{\RR}{\mathbb{R}^2}
\newcommand{\ZZ}{\mathbb{Z}^2}
\newcommand{\hTTf}[1][f]{\mathbf{L}_{#1}}

\newcommand{\cF}{\mathcal{F}}

\newcommand{\tf}{\tilde{f}}

\newcommand{\tx}{\tilde{x}}
\newcommand{\ty}{\tilde{y}}

\newcommand{\tmu}{\tilde{\mu}}

\newcommand{\bomega}{\boldsymbol{\omega}}

\newcommand{\hf}{\hat{f}}

\newcommand{\bmu}{\boldsymbol{\mu}}
\newcommand{\tbmu}{\tilde{\bmu}}
\newcommand{\hmu}{\hat{\mu}}

\newcommand{\hx}[1][x]{\hat{#1}}
\newcommand{\hy}{\hat{y}}

\newcommand{\teta}{\tilde{\eta}}
\newcommand{\tbeta}{\tilde{\boldsymbol{\eta}}}

\newcommand{\Mat}{\mathrm{M}_{2 \times 2}(\mathbb{Z})}
\newcommand{\reg}{\mathcal{R}}

\DeclareMathOperator{\GL}{GL}

\DeclareMathOperator*{\essinf}{ess\,inf}

\newcommand{\pcov}{\pi_{\mathrm{cov}}}
\newcommand{\pext}{\pi_{\mathrm{ext}}}

\newcommand{\Deh}{\Delta_{\alpha}^h}
\newcommand{\Dev}{\Delta_{\alpha}^v}

\newcommand{\Sol}{\operatorname{Sol}}
\newcommand{\proj}{\operatorname{proj}}

\newcommand{\fsol}{Sf}
\newcommand{\bx}{\mathbf{x}}

\newcommand{\Wuloc}[1]{W^{u}_{\mathrm{loc}}(#1)}
\newcommand{\Wsloc}[1]{W^{s}_{\mathrm{loc}}(#1)}

\newcommand{\hWuloc}[1]{\hat{W}^{u}_{\mathrm{loc}}(#1)}
\newcommand{\hWsloc}[1]{\hat{W}^{s}_{\mathrm{loc}}(#1)}

\newcommand{\Leb}{\mathrm{Leb}}

\hypersetup{
	unicode=true,          
	pdftoolbar=true,        
	pdfmenubar=true,        
	pdffitwindow=false,     
	pdfstartview={FitH},    
	pdftitle={Non-uniformly hyperbolic endomorphisms},    
	pdfauthor={M. Andersson, P. D. Carrasco and R. Saghin},     
	pdfkeywords={Non-Uniformly Hyperbolic, Stably Ergodic, Continuity of Lyapunov exponents}, 
	pdfnewwindow=true,      
	colorlinks=true,       
	linkcolor=blue,     
	citecolor=green,       
	filecolor=magenta,     
	urlcolor=cyan          
}

\makeatletter
\newcommand{\subjclass}[2][2020]{%
  \let\@oldtitle\@title%
  \gdef\@title{\@oldtitle\footnotetext{#1 \emph{MSC2020 classification:} #2}}%
}
\newcommand{\keywords}[1]{%
  \let\@@oldtitle\@title%
  \gdef\@title{\@@oldtitle\footnotetext{\emph{Keywords:} #1.}}%
}
\makeatother

\begin{document}

 \title{Non-uniformly hyperbolic endomorphisms}
  \author[1]{Martin Andersson \thanks{Partially supported by AMSUD220029}} 	
  \author[2]{Pablo D. Carrasco \thanks{Partially supported by CNPq 403041/2021-0 AMSUD220029}}
  \author[3]{Radu Saghin \thanks{Partially supported by Fondecyt 1230632, AMSUD220029}}

\affil[1]{GMA-UFF, Rua Professor Marcos Waldemar de Freitas Reis, S/Nº
Campus do Gragoatá, Blocos G e H, Niterói – RJ, BR24210-201}
 \affil[2]{ICEx-UFMG, Avda. Presidente Antonio Carlos 6627, Belo Horizonte-MG, BR31270-90}
 \affil[3]{Pontificia Universidad Cat\'olica de Valpara\'{\i}so, Blanco Viel 596,
 Cerro Bar\'on, Valpara\'{\i}so-Chile.}

  \date{\today}
  
  \maketitle

%\clearpage
\selectlanguage{english}
\begin{abstract}

We show the existence of large $\mathcal C^1$ open sets of area preserving endomorphisms of the two-torus which have no dominated splitting and are non-uniformly hyperbolic, meaning that Lebesgue almost every point has a positive and a negative Lyapunov exponent. The integrated Lyapunov exponents vary continuously with the dynamics in the $\mathcal C^1$ topology and can be taken as far away from zero as desired. Explicit real analytic examples are obtained by deforming linear endomorphisms, including expanding ones. The technique works in nearly every homotopy class and the examples are stably ergodic (in fact Bernoulli), provided that the linear map has no eigenvalue of modulus one.
\end{abstract}

%%%%%%%%%%%%%%%%%%%%%%%%%%%%%%%%%%%%%%%%%%%%%%%%%%%%%%%%%%%%%%%%%%%%%%%%%%%%%%%%%%%%%%%%%%%%%%%%%%%%%%%%%%%%%%%%%%%%
\section{Introduction and Results} % (fold)
\label{sec:introduction_and_results}
%%%%%%%%%%%%%%%%%%%%%%%%%%%%%%%%%%%%%%%%%%%%%%%%%%%%%%%%%%%%%%%%%%%%%%%%%%%%%%%%%%%%%%%%%%%%%%%%%%%%%%%%%%%%%%%%%%%%

In this article we deal with conservative maps on the two-torus $\Tor$, that is, maps preserving the Lebesgue (Haar) measure $\mu$. In smooth ergodic theory one is typically interested in the asymptotic behavior exhibited by the system in consideration, and a basic (but nonetheless, powerful) mechanism for the analysis is provided by the Lyapunov exponents. For completeness we will recall the definition below, but we assume some familiarity of the reader with these concepts, as covered for example in \cite{Barreira2013}. We will, however, focus in the two dimensional case and discuss only this setting.

\begin{definition}
For a differentiable map $f:\Tor\to\Tor$ and a vector $(x,v)\in T\Tor$ with $v\neq 0$, the quantity
\begin{align*}
\bar{\chi}(x,v)=\limsup_{n\to\infty}\frac{\log\norm{D_xf^n(v)}}{n}
\end{align*}
is the Lyapunov exponent of $f$ at $(x,v)$.
\end{definition}
Classical results (as for example \cite{FurstKest60}) show that one can find a full area set $M'$ such that for points $x\in M'$ the previous limit exists (we denote it as $\chi(x,v)$) for every $v\neq 0$, and coincides with either of the following two quantities\footnote{For a matrix $A$, $\conorm{A}=\inf\{\norm{Av}:\norm{v}=1\}$ denotes its conorm.}:
\begin{align*}
\chi^+(x)&=\lim_{n\to+\infty}\frac{\log\norm{D_xf^n}}{n}\\
\chi^-(x)&=\lim_{n\to+\infty}\frac{\log\conorm{D_xf^n}}{n}.
\end{align*}
In fact, due to Oseledet's theorem \cite{Oseledets} there exists a measurable bundle $E^-$ defined in a full measure set $M_0$ such that for $x\in M_0, 0\neq v\in E^{-}(x)\Rightarrow \chi(x,v)=\chi^{-}(x), v\in \RR\setminus E^{-}(x)\Rightarrow \chi(x,v)=\chi^{+}(x)$. One can check that
\[
	\int (\chi^{+}(x)+\chi^{-}(x)) d\mu(p)=\int \log |\det D_xf| d\mu(x)>0
\]
so $\chi^{+}(x)>0$ almost everywhere.

\begin{definition}
The map $f$ is non-uniformly hyperbolic (NUH) if $\chi^-<0<\chi^+$ $\mu$-almost everywhere.
\end{definition}

Non-uniform hyperbolicity means that there is exponential divergence of orbits, both in the past and in the future. This notion, introduced by Y. Pesin, generalizes classical hyperbolicity, and it is difficult to overemphasize its importance in modern dynamics. 

For NUH diffeomorphisms, a basic feature is that even though they can be found on any surface \cite{BerSurf} (and indeed, on any closed manifold of dimension larger than one \cite{EveryCompHyp}), typically they are fragile in the following sense: any conservative surface diffeomorphism $f$ can be approximated in the $\mathcal{C}^1$ topology by a diffeomorphism having zero exponents, unless $f$ is Anosov. This is consequence of the Bochi-Ma\~n\'e theorem \cite{BOCHI2002}, and reveals a rigidity phenomena for such systems.

It is remarkable that the corresponding question for endomorphisms has not been addressed before, but this can be explained by the surplus of additional difficulties that the non-invertible case presents, leading to a much less developed theory than its corresponding invertible counterpart. The aim of this paper is to make a contribution to the ergodic theory of non-invertible endomorphisms of $\Tor$, and present interesting new phenomena that perhaps will motivate additional research.

In contrast to the case of diffeomorphisms, we will establish the existence of $\mathcal C^1$ robust NUH endomorphisms in $\Tor$, without essentially any topological obstruction. These examples share several similarities with hyperbolic diffeomorphisms. On the one hand, their integrated Lyapunov exponents (and therefore their exponents in the ergodic case) depend continuously with the map. We point out to the reader that, due to the Bochi-Mañé theorem cited before, this is not true for diffeomorphisms except for Anosov ones. On the other hand, assuming transitivity of their linear part (which can be easily checked by looking at the eigenvalues of the associated matrix), they are ergodic (even Bernoulli, see \ref{sub:stable_ergodicity}), and they remain so under $\mathcal C^2$ perturbations. Therefore, the class of systems introduced in this article seems to have enough robustness to permit a throughout analysis, and we hope to motivate the reader to explore this new realm.

\subsection{Existence of robust NUH endomorphisms} % (fold)
\label{sub:existence_of_persistent_nuh_endomorphisms}

Any map $f:\Tor\to\Tor$ is homotopic to a linear endomorphism $E$, induced by an  integer matrix that we denote by the same letter. In this work, we are interested in non-invertible local diffeomorphisms and will therefore only consider the case when $|\det E | \geq 2$. Considering the projective action $\Pp E: \Pp \RR \to \Pp \RR$ that induces $E$, we have three possibilities:
\begin{itemize}
	\item $\Pp E=\mathrm{Id}$, or
	\item there exists a parabolic fix point $p\in \Pp \RR$, that is repelling on one side and attracting on the other, or
	\item there exist two hyperbolic fix points $p,q\in \Pp \RR$, and the dynamics is of type north pole - south pole, or
	\item $\Pp E$ is conjugated to a rotation.
\end{itemize}

Note that the first corresponds to the case when $E$ is an homothety.

If $f:\Tor\to\Tor$ is of class $\mathcal C^1$ we consider the class of $\mathcal C^1$ smoothly homotopic maps to $f$,
\[
	[f]=\{g:\Tor\to\Tor: \text{there exists a }\mathcal C^1\text{ homotopy between $f$ and $g$}\},
\]
We denote by $\mathrm{End}_{\mu}^r(\TT)$ the set of $\mathcal C^r$ local diffeomorphisms of $\TT$ preserving the Lebesgue measure $\mu$, that are not invertible. For $f\in \mathrm{End}^1_{\mu}(\TT), (x,v)\in T^1\TT$ we define 
\begin{equation}\label{eq:I}
I(x,v;f^n)=\sum_{y\in f^{-n}(x)}\frac{\log\|(D_yf^n)^{-1}v\|}{|\det(D_yf^n)|}
\end{equation}
and let
\begin{equation}\label{eq:Cf}
C_{\chi}(f)=\sup_{n\in\mathbb N}\frac 1n\inf_{(x,v)\in T^1\TT}I(x,v;f^n).
\end{equation}
The set
\begin{equation}\label{eq:U}
\mathcal U:=\{f\in \mathrm{End}^1_{\mu}(\TT):\ C_{\chi}(f)>0\}
\end{equation}
is open in the $\mathcal C^1$ topology: indeed, if $f\in\mathcal U$ then for some $n$ it follows $\frac 1n\inf_{(x,v)\in T^1\TT}I(x,v;f^n)>\frac{C_{\chi}(f)}{2}$, hence taking a sufficiently small $\mathcal C^1$ neighborhood $\mathcal V$ of $f$ we can guarantee that any $g\in\mathcal V$ is an endomorphisms, and satisfies $\frac 1n\inf_{(x,v)\in T^1\TT}I(x,v;g^n)>\frac{C_{\chi}(f)}{4}$, thus impliying that $g\in \mathcal U$.

Our first main result is the following.

\begin{theoremA}\hypertarget{theoremA}{}
Any $f\in \mathcal U$ is non-uniformly hyperbolic. Moreover, if $E=(e_{ij})\in\Mat$ is not a homothety and $|\det E|/\gcd{e_{ij}}>4$, then the intersection $[E]\cap \mathcal U$ is non-empty, and in fact contains maps that are real analytically homotopic to $E$. 
\end{theoremA}

This in particular establishes the existence of $\mathcal C^1$ open sets of NUH endomorphisms. The above is a flexibility theorem: different from the diffeomorphism case, there are no topological obstructions for the existence of robust NUH endomorphisms in $\Tor$. We also point out this is the first robustly NUH (area preserving) example of endomorphism, that it is not expanding, and does not admit an invariant distribution.

\begin{remark}\label{rem:hipotesisteoA}
If $\tau_1(E)=\gcd{e_{ij}}, \tau_2(E)=|\det E|/\tau_1(E)$,  then the Smith Normal Form of $E$ is the diagonal matrix with entries $\tau_1(E)$ and $\tau_2(E)$ (see subsection \ref{ss:Smith}). The condition on $\tau_2(E)$ from \hyperlink{theoremA}{Theorem A} is valid for all pairs $(\tau_1(E),\tau_2(E))$ with the exception of those in
 \[
 	\{(1,2), (1,3), (1,4), (2, 2), (2,4), (3,3), (4,4)\}.
 \] 
This requirement on $\tau_2(E)$ and the non-homothety assumption are made to avoid some technical difficulties,
 and we expect that with additional work this restriction can be lifted to any (non-invertible) endomorphism; see Subsection \ref{subs:furtherQ}. Let us also point out that our methods give effective bounds for the $\mathcal{C}^1$ distance between constructed elements $f\in [E]\cap \mathcal U$ and $E$.
\end{remark}

Observe the following curious consequence.

\begin{corollary}
If $F:\Tor\to\Tor$ is a $\mathcal{C}^1$ uniformly expanding map of degree larger than $16$ that is not homotopic to a homothety, then it is a factor (semi-conjugate) of a smooth conservative map $f$ that is non-uniformly hyperbolic with one positive and one negative Lyapunov exponent.
\end{corollary}

\begin{proof}
By \cite{shubendo} $F$ is conjugate to its linear part $E$, and $E$ is a factor of the map $f$ constructed in \hyperlink{theoremA}{Theorem A}.
\end{proof}

For any such expanding map, necessarily both its exponents are positive (in fact, with respect to any of its invariant measures). Nevertheless, it can be obtained as a quotient of a system which has contracting direction at Lebesgue almost every of its phase space.

\subsection{Continuity of the exponents}

A natural question which arises is how do the Lyapunov exponents depend on the map $f$. Since we do not know a priori that $f$ is ergodic, it makes sense to ask what is the regularity of the maps $f\mapsto\int_{\Tor}\chi^-(f)d\mu$, $f\mapsto\int_{\Tor}\chi^+(f)d\mu$, for $f\in \mathrm{End}^r_{\mu}(\TT)$, $r\geq 1$.

The classical results of Ma\~ ne-Bochi-Viana show that one cannot expect continuity in the $\mathcal C^1$ topology for diffeomorphisms without dominated (Oseledets) splitting. Continuity in finer topologies is known for certain restricted classes of diffeomorphisms, but the general panorama is still not clear. The reader is referred to \cite{Liang2018} and references therein for a discussion in the diffeomorphism (partially hyperbolic) setting. For continuity of Lyapunov exponents for cocycles over homeomorphisms the reader can check the survey \cite{VIANA2018}.

On the other hand, the results in \cite{expcontendoVianaYang} for cocycles over endomorphisms suggest that we may be able to obtain better regularity than in its invertible counterpart. Indeed, we obtain the following result.

Define
\[
C_{\det}(f)=\sup_{n\in\mathbb N}\frac 1n\inf_{x\in M}\log|\det(Df_x^n)|>0
\]
and
\begin{equation}\label{eq:U1}
\mathcal U_1=\left\{ f\in \mathrm{End}^1_{\mu}(\TT): C_{\chi}(f)>-\frac 12C_{\det}(f)\right\},
\end{equation}
where $C_{\chi}(f)$ is defined in \eqref{eq:Cf}. Clearly $\mathcal U_1$ is open and contains the set $\mathcal U$ of NUH endomorphisms constructed in \hyperlink{theoremA}{Theorem A}.

\begin{theoremB}\hypertarget{theoremB}{}
The maps $\mathcal U_1\ni f\mapsto\int_{\Tor}\chi^-(f)d\mu$, $\mathcal U_1\ni f\mapsto\int_{\Tor}\chi^+(f)d\mu$ are continuous in the $\mathcal C^1$ topology.
 \end{theoremB}
 
 An immediate consequence of Pesin formula for endomorphisms \cite{pesinendo} is the result below.
 
 \begin{corollary}
The map $\mathcal U\ni f\mapsto h_{\mu}(f)$ is continuous in the $\mathcal C^1$ topology. Here $h_{\mu}(f)$ denotes the metric entropy of $f$ with respect to $\mu$.
  \end{corollary}

\subsection{Stable ergodicity} % (fold)
\label{sub:stable_ergodicity}

Some of the examples constructed can be shown to be stably ergodic.

\begin{definition}
We say that an endomorphism $f$ is stably ergodic (for $\mu$) if there exists a $\mathcal C^1$ neighborhood $\mathcal U'$ of $f$ in $\mathrm{End}^1_{\mu}(\TT)$ so that every $g\in \mathcal{U'}$ of class $\mathcal{C}^{2}$ is also ergodic.
\end{definition}

The study of stably ergodic diffeomorphisms (that are not uniformly hyperbolic) was initiated by the results of Grayson, Pugh and Shub \cite{StErgGeo}, and was promoted by Pugh and Shub to a theory which in the present attracts a large amount of research. Citing all these bibliography would take us astray, so we will content indicating some significative recent articles and refer the reader to them for additional information: \cite{ErgPH,AccStaErg,Avila2021}.

Despite its popularity in the context of invertible dynamics, stable ergodicity for endomorphisms has never been given serious attention. In fact, we are unaware of any earlier attempt to give examples outside the realm of uniformly expanding systems. On the other hand, there are quite a few studies of physical measures for endomorphisms, and some of these can be translated into results about stable ergodicity. Most notably, for maps that are non-uniformly expanding in the sense of \cite{SRBMostexp}, the basins of physical measures are open. In a conservative context, that translates to saying that ergodic components of Lebesgue measure are open (modulo sets of zero Lebesgue measure). Therefore, robust almost everywhere non-uniform expansion together with robust transitivity implies stable ergodicity. 

For partially hyperbolic endomorphisms we have a very complete picture in the dissipative setting thanks to the seminal work of Tsujii \cite{Tsujii2005}: $\mathcal C^2$-generic partially hyperbolic local diffeomorphisms on $\TT$ have a finite number of physical measures whose basins cover a set of full Lebesgue measure of the manifold. One of the curious facts is that these SRB measures are absolutely continuous with respect to Lebesgue, even if the central Lyapunov exponent is slightly negative. In the recent \cite{Tsujii2022} Tsujii and Zhang construct smooth families of endomorphisms that are non-uniformly hyperbolic with respect to the SRB.  These works seem to require regularity higher than $\mathcal C^2$. In the conservative setting it is likely that one can generalize the concept of accessibility to partially hyperbolic endomorphisms (by looking at the inverse limit and consider $\pext(x)$ as the strong stable manifold of $x$) and then adapt the Grayson-Pugh-Shub programme to obtain stable ergodicity.

For partially hyperbolic endomorphisms having only admissible measures (the analogue of Gibbs $u$-states for non-invertible maps) with negative central Lyapunov exponents, the first author proved in \cite{Andersson2010} that having a single physical measure is a $\mathcal C^2$ open property. For conservative maps that translates into saying that ergodicity implies ($\mathcal C^2$) stable ergodicity. We find it very likely that such maps are indeed $\mathcal C^1$ stably ergodic, meaning that every $\mathcal C^2$ map $\mathcal C^1$ close is also ergodic. This was proved by Yang \cite{Yang2021} in the analogous setting for diffeomorphisms, and we see no reason why his proof should not work for endomorphisms as well. 

In this work we produce examples of stably ergodic endomorphisms of a rather exotic kind: robustly non-uniformly ergodic maps without any dominated splitting. This stands in sharp contrast the invertible case. Indeed, it was shown by Bochi, Fayad and Pujals \cite{BoFaPu2006} that any stably ergodic diffeomorphism can be approximated by one having a dominated splitting.

Our contribution to the theory is the following.

\begin{theoremC}\hypertarget{theoremC}{}
For any linear endomorphism $E$ as in \hyperlink{theoremA}{Theorem A}, if $\pm 1$ is not an eigenvalue of $E$ then 
$[E]\cap \mathcal U$ contains stably ergodic endomorphisms.

In fact, $[E]\cap \mathcal U$ consists of stably Bernoulli\footnote{An endomorphism is Bernoulli if its inverse limit lift is measure theoretically isomorphic to a Bernoulli process.} endomorphisms (and in particular, maps that are mixing of all orders).
\end{theoremC}

Thus, we obtain examples of stably ergodic non-uniformly hyperbolic endomorphisms that do not admit any invariant distribution, and therefore are not partially hyperbolic. 

\smallskip 

\hyperlink{theoremC}{Theorem C} above will be obtained as consequence of the following general result, that is of independent interest. For the discussion of global stable manifolds and Pesin theory see Section \ref{sec:Pesintheory}.

\begin{theoremD}\hypertarget{theoremD}{}
Let $f$ be a $\mathcal C^2$ transitive area preserving non-uniformly hyperbolic endomorphism on a surface. Suppose there exists $\lambda>0$ such that  for almost every point,  the diameter\footnote{The diameter is measured inside the ambient space.} of the global stable manifold of $x$ is larger than $\lambda$. Then $f$ is ergodic.

If furthermore $f^n$ is transitive for any positive natural number $n$, then $f$ is Bernoulli.
\end{theoremD}

\subsection{Comments on the proofs}

We will make some comments on the ideas of the proofs involved in the article.

A fundamental tool is the natural extension of the endomorphism $f$: it is the space $\hTTf$ of the backward orbits, with a natural projection $\pi_{ext}$ to the torus. The endomorphism $f$ lifts to an invertible map $\hat f$ and the area $\mu$ lifts to an invariant measure $\hat\mu$. The Lyapunov exponents of $f$ coincide with the Lyapunov exponents of the $Df$ cocycle lifted to $\hTTf$. An important observation is the fact that $\lim_{n\rightarrow\infty}\frac{1}{n}I(x,v;f^n)$ is exactly the integral of the Lyapunov exponents of the lifted cocycle at $(\hat x,v)$ on the fiber $\pi_{ext}^{-1}(x)$ with respect to the disintegration of $\hat\mu$ on $\pi_{ext}^{-1}(x)$. This will eventually imply that any map in $\mathcal U$ must be non-uniformly hyperbolic. 

Let us mention that the property defining $\mathcal U$ has a similar counterpart in the setting of random maps, namely the ``uniform expansion property'' of Dolgopyat-Krikorian \cite{DK07}. This property has several nice consequences for the properties of stationary measures, and was investigated in a series of papers like \cite{DK07}, \cite{Brown2017}, \cite{Ch20}, \cite{Liu16}.

The basic ideas in the construction of the specific examples in the homotopy class of $E$ are the following. We modify $E$ such that we make the map (strongly) hyperbolic in some ``good''  region $\mathcal G$, which is chosen to be large enough so that most of the backward orbits spend a substantial time inside it. In general this may not be enough in order to obtain hyperbolicity without some further control, because even if one obtains a relatively good expansion on some large piece of orbit, it is possible that once we loose the control, the next entry in the good region is such that the expanding direction goes to the contracting direction, and thus all the hyperbolicity which we gained gets lost. In order to avoid this we have to make some kind of non-degeneracy assumption, in particular we construct the map with slightly different hyperbolicity in different parts of the good region $\mathcal G$. Thus we may loose the hyperbolicity at some particular orbits, but we will maintain it at other ones, and this turns out to be enough in order to obtain the conclusion.

The proof of the continuity of the exponents is based on methods developed for the study of the Lyapunov exponents for cocycles with holonomies. The definition of $\mathcal U_1$ guarantees that the Lyapunov spectrum is simple at almost every point. One considers the projectivization of the cocycle over $\hat f$, and the invariant measure $\hat\mu$ lifts to convex combinations of Dirac measures supported at the Oseledets subspaces at almost every point. Assume that a sequence $f_n$ converges to $f$ in the $\mathcal C^1$ topology. The only way the continuity of the exponents can fail is if a sequence of stable lifts of $\hat\mu$ for $f_n$ (lifts supported at the subspaces corresponding to the smallest Lyapunov exponent) does not converge to the stable lift of $\hat\mu$ for $f$, in other words the limit contains some parts of the unstable lift of $\hat\mu$ for $f$. The stable lifts are invariant under the holonomy given by the projection $\pi_{ext}$, and this property can be passed to the limit. This would imply that there is a nontrivial part of the unstable lift of $\hat\mu$ for $f$ which is invariant under the $\pi_{ext}$ holonomy. It turns out that this is impossible because of the definition of $\mathcal U_1$, this definition already implies that the (strong) unstable Oseledets subspaces cannot be independent of the past, so they cannot be constant on $\pext$ fibers. This condition is reminiscent of the ``non-random stable bundle'' case appearing in the classification theorem of A. Brown and F. Rodriguez-Hertz \cite{Brown2017} of hyperbolic stationary measures for random diffeomorphisms of surfaces. 

For the proof of the stable ergodicity we rely on the classical Hopf argument. In order to obtain the intersections between the stable and unstable manifolds we introduce a new method: we first prove that the stable manifolds at almost every points have large diameter. This is different from the classical notion of ``large Pesin manifolds'', because we do not have control on how the manifolds turn inside the torus. However we can use ideas from Rodriguez Hertz-Rodriguez Hertz-Tahzibi-Ures \cite{Hertz2011} to show that in the case of endomorphisms on surfaces, large stable manifolds in our sense do imply that the ergodic components are open modulo zero sets. In order to conclude the proof we use a result of Andersson \cite{Andersson2016} on the transitivity of area preserving endomorphisms on $\TT$.

\subsection{Further questions}\label{subs:furtherQ}

Here we make some comments and we ask some further questions related to this work.
\begin{enumerate}
\item {\bf Non-uniformly hyperbolic endomorphisms in {\it all} homotopy classes.}

We proved that the $\mathcal{C}^1$ open set $\mathcal U$ of non-uniformly hyperbolic endomorphisms intersects most of the homotopy classes of endomorphisms.
\begin{question}
Is it true that $\mathcal U$ intersects all the homotopy classes of endomorphisms on $\TT$?
\end{question}
It is reasonable to expect that this is true, however if the endomorphism $E$ is a homothety then one has to construct a more complicated good hyperbolic region $\mathcal G$. Also if the degree of $E$ is very small then one should take into consideration the position of several preimages of every point, in order to obtain that most of them fall into the good region.

\item {\bf Stable ergodicity in a more general setting}.

We can obtain the stable ergodicity of the non-uniformly hyperbolic endomorphisms under various restrictions. One of them is a restriction of the homotopy class, the eigenvalues have absolute value different from one, in order to apply Andersson's result on transitivity of endomorphisms.
\begin{question}
Are there stably ergodic non-uniformly hyperbolic endomorphisms in every homotopy class of endomorphisms on $\TT$?
\end{question}
Another restriction comes from the proof of the existence of large stable manifolds, we need that the good region is large enough with good hyperbolicity while the critical region is small. On the other hand, due to the area expansion, the unstable manifolds of a non-uniformly hyperbolic endomorphism are always large.
\begin{question}
Is it true that every transitive non-uniformly hyperbolic endomorphism on a surface is stably ergodic?
\end{question}
Our methods cannot be applied directly for large unstable manifolds, because they do not form a lamination, so one cannot obtain immediately intersections of stable and unstable manifolds. One does obtain however intersections between different unstable manifolds, and this may be enough in order to obtain the ergodicity with some further work.
If the Jacobian of the endomorphism is constant the proof of the ergodicity is straightforward from the absolute continuity of the $\pi_{ext}$ holonomy, so the difficult question is when the Jacobian is not constant.
\item {\bf More continuity of the exponents.}

In general the continuity of the exponents may fail if parts of the stable lift of the measure to the projectivization of the cocycle move in the limit to parts of the unstable lift. But this would imply that the strong unstable subspaces are invariant under the $\pi_{ext}$ holonomy, and this seems to be a rigid situation. 
\begin{question}
Is it true that the average Lyapunov exponents are continuous in (an open dense subset of) the space of $\mathcal C^1$ area preserving endomorphisms of surfaces?
\end{question}
One may also ask how do the Lyapunov exponents depend on the invariant measure, not necessarily Lebesgue.
\begin{question}
Can we say something about the continuity of the Lyapunov exponents of maps in $\mathcal U_1$ with respect to the invariant measure, in a neighborhood of the Lebesgue measure?
\end{question}

\item {\bf Flexibility of Lyapunov exponents for endomorphisms on $\TT$}. 

Bochi, Katok and Rodriguez Hertz proposed in \cite{Bochi2019} a program regarding the flexibility of Lyapunov exponents for diffeomorphisms. The same questions can be addressed in the case of endomorphisms:

\begin{question}[Weak flexibility on $\TT$]
Given any two real numbers $\lambda_1,\lambda_2$ with $\lambda_1+\lambda_2> 0$, does there exist an ergodic measure preserving (smooth) endomorphism of $\TT$ with Lyapunov exponents $\lambda_1$ and $\lambda_2$?
\end{question}

\begin{question}[Strong flexibility on $\TT$]
Given a linear endomorphism $E$ of $\TT$ and $\lambda_1,\lambda_2\in \RR$ with $0<\lambda_1+\lambda_2$ $\leq\log(|\det(E)|)$, does there exist an ergodic measure preserving (smooth) endomorphism of $\TT$, isotopic to $E$, with Lyapunov exponents $\lambda_1$ and $\lambda_2$?
\end{question}

Our results answer the second question for many isotopy classes and for $\lambda_1<<0<<\lambda_2$, $\lambda_1+\lambda_2$ close to $\log(|\det(E)|)$. Any progress in the Questions 1-5 above would also contribute to the flexibility program for endomorphisms on $\TT$. 

\item {\bf Dissipative examples.}

In our paper we consider endomorphisms which preserve the Lebesgue measure on the torus. Similar consideration can also be made for other invariant measures, once we have some control on the disintegrations of the lift to the natural extension along the fibers of $\pi_{ext}$. A natural question is the following.
\begin{question}
Can we deduce the nonuniform hyperbolicity and ergodicity for other meaningful invariant measures of our examples and dissipative perturbations of them? In particular nonuniform hyperbolicity and uniqueness of equilibrium states for some specific potentials?
\end{question}
Of special interest is the existence and uniqueness of SRB measures.
\begin{question}
Can we construct (robust) examples of dissipative endomorphisms on surfaces having (unique) non-uniformly hyperbolic SRB measures and no dominated splitting?
\end{question}
\item {\bf Weak partially hyperbolic diffeomorphisms with two-dimensional center-unstable bundle.}

Our examples can be seen (in the natural extension) as weak partially hyperbolic attractors of diffeomorphisms having a (super-)strong stable bundle and a center-unstable bundle of dimension two. So it is natural to ask how much our results can be transferred to the setting of weak partially hyperbolic diffeomorphisms with a two-dimensional area-expanding center-unstable bundle (or with a two-dimensional area-contracting center-stable bundle, taking the inverse map).

\begin{question}
Can one construct examples of $\mathcal C^2$ stably ergodic non-uniformly hyperbolic diffeomorphisms with a stable bundle and a two-dimensional center-unstable bundle, such that the center-unstable bundle has a positive and a negative exponent and no dominated (sub) splitting? How about dissipative examples with SRB measures?
\end{question}

Somehow similar examples of non-uniformly hyperbolic diffeomorphisms with two-dimensional center and both stable and unstable bundles were constructed in \cite{NUHD}, and ergodicity was proven in \cite{Obata}. 

The stable holonomy for the natural extension of endomorphisms is given by $\pi_{ext}$ which is very nice, and we make strong use of this fact in our proofs. In the case of weak partially hyperbolic diffeomorphisms one would have to assume better regularity of the map and maybe some bunching condition, in order to obtain some good properties of the stable holonomy. 
\item {\bf Higher dimensional examples.}

Another natural question is wether similar examples can be constructed for higher dimensional endomorphisms, with a given number of positive and negative exponents. Furthermore, it is interesting to see if such examples can be adapted to weak partially hyperbolic diffeomorphisms.
\end{enumerate}

\subsection{Organization of the article}

The rest of the article is organized as follows. In the next section we introduce some notation and discuss some prerequisites for the rest of the work; in particular we review the concept of inverse limit and prove that every endomorphism in $\mathcal U$ must be non-uniformly hyperbolic (Proposition \ref{pro:existeexpneg}). The third section is devoted to the proof of the second part of \hyperlink{theoremA}{Theorem A}, we construct the examples of maps in $[E]\cap\mathcal U$ for most homotopy classes $[E]$. At the end of it we include a specific computation of an example to convince the reader in the controllability of all constants. This example is a so called ``expanding'' standard map, which has additional interest and has received some attention lately. The fourth section contains the proof of \hyperlink{theoremB}{Theorem B}, and it includes a discussion about measures with product disintegrations.  In the fifth section we review some facts of Pesin theory for our context, in particular the notion of stable and unstable manifolds for endomorphism. This is used in the last part in order to establish the stable ergodicity and the Bernoulli property of the examples. In the \hyperlink{sec:appendix}{Appendix} we discuss the inverse limit space, and how the area lifts to an invariant measure on it.

\paragraph{Acknowledgements:} We thank the anonymous referee for the useful comments and suggestions.

\subsection{Added in proof}

There has been some advances while this article was under review that are worth mentioning. Concerning Question $1$ of \ref{subs:furtherQ}, it has been proven by V. Janeiro \cite{Janeiro2023} that for degree $\geq 5^2$ the requirement on the linear part not being an homothety is not necessary. On the other hand, S. Ramirez and K. Vivas \cite{Ramirez2023} improved the requirement on the degree to in the non-homothety case, requiring only that the determinant is (in absolute value) larger than $2$. This paper also addresses Question $2$, considering the case when the linear part is not transitive, provided that the degree is large enough. Complementary to this, the work of Y. Lima, D. Obata and M. Poletti \cite{Lima2024} shows that under the transitivity condition, every element in $\mathcal U$ is stably ergodic (Bernoulli). The later paper also deals with Question $8$, establishing uniqueness of the measure of maximal entropy, in some subset of $\mathcal U$.

% section introduction_and_results (end)

%%%%%%%%%%%%%%%%%%%%%%%%%%%%%%%%%%%%%%%%%%%%%%%%%%%%%%%%%%%%%%%%%%%%%%%%%%%%%%%%%%%%%%%%%%%%%%%%%%%%%%%%%%%%%%%%%%%%%%%%
\section{Preparations and preliminary material} % (fold)
\label{sec:preparations_and_preliminary_material}
%%%%%%%%%%%%%%%%%%%%%%%%%%%%%%%%%%%%%%%%%%%%%%%%%%%%%%%%%%%%%%%%%%%%%%%%%%%%%%%%%%%%%%%%%%%%%%%%%%%%%%%%%%%%%%%%%%%%%%%%
%%%%%%%%%%%%%%%%%%%%%%%%%%%%%%%%%%%%%%%%%%%%%%%%%%%%%%%%%%%%%%%%%%%%%%%%%%%%%%%%%%%%%%%%%%%%%%%%%%%%%%%%%%%%%%%%%%%%%%%%
%%%%%%%%%%%%%%%%%%%%%%%%%%%%%%%%%%%%%%%%%%%%%%%%%%%%%%%%%%%%%%%%%%%%%%%%%%%%%%%%%%%%%%%%%%%%%%%%%%%%%%%%%%%%%%%%%%%%%%%%

In this part we collect some facts that will be used in the proofs, and we will show that the endomorphisms in $\mathcal U$ are non-uniformly hyperbolic. 

\subsection{Natural extension}\label{sub:naturalext}

Given a continuous map $f: \TT \to \TT$, we consider its natural extension
\[\hTTf = \{ (x_0, x_1, x_2, \ldots) \in (\TT)^{\Z_+}: f(x_{i+1}) = x_{i} \ \forall i \geq 0 \},\]
endowed with the (weak) product topology inherited from $(\TT)^{\N}$. We denote by $\pext: \hTTf \to \TT$ the projection onto the first coordinate.
The map
\begin{align*}
\hf: \hTTf & \to \hTTf \\
 (x_0, x_1, x_2, \ldots) & \mapsto (f(x_0), x_0, x_1, \ldots).
\end{align*}
is then a homeomorphism on $\hTTf$ satisfying $\pext \circ \hf = f \circ \pext$. In the case where $f$ is a self-cover (as is the case for local diffeomorphisms), the set $\hTTf$ has an interesting structure: it is a solenoidal $2$-manifold, meaning that it is locally homeomorphic to the product of a $2$-dimensional disk with a Cantor set. More specifically, $\hTTf$ is a fiber bundle whose fibers $\pext^{-1}(x)$ are Cantor sets. This structure will be of importance when proving continuity of Lyapunov exponents (section \ref{sec:continuity}) and is explained in some detail in the \hyperlink{sec:appendix}{Appendix}. 

Given any $f$-invariant Borel measure $\mu$ on $\TT$, there exists a unique $\hf$-invariant Borel measure $\hmu$ on $\hTTf$ such that $(\pext)_* \hmu = \mu$ called the $\emph{lift}$ of $\mu$. According to Rokhlin's disintegration theorem, $\hmu$ can be disintegrated along the partition into fibers $\{\pext^{-1}(x): x \in \TT \}$, yielding a family of conditionals. These conditionals are determined by their measures on cylinders in the following way: 
\begin{proposition}\label{pro:desintegracionhmu}
Let $\mu$ be the Haar measure on $\TT$ and $\hmu$ its lift on $\hTTf$. Then
\begin{equation}\label{muhat}
\hmu(A) = \int \hmu_x(A) \ d\mu(x)
\end{equation}
for every Borel set $A \subset \hTTf$, where $\hmu_x$ is the unique measure on $\pext^{-1}(x)$ satisfying
\begin{equation}
\hmu_x(\{(\xi_0, \xi_1, \ldots) \in \pext^{-1}(x): \xi_i = x_i\}) = | \det Df^i(x_i)|^{-1}.
\end{equation}
\end{proposition}

We refer to the \hyperlink{sec:appendix}{Appendix} for a proof.

\subsection{Lyapunov exponents for the natural extension}\label{sub:lyapunov_exponents_for_the_natural_extension}

Given a cocycle $A:X\to \GL_2(\Real)$ over a map $h:X\to X$ we denote, for $x\in X, v\in\RR\setminus\{0\}$,
\begin{align*}
\chi_{(h,A)}(x,v)=\lim_{n\to+\infty}\frac{\log\norm{A^{(n)}(x)\cdot v}}{n}\\
\chi_{(h,A)}^+(x)=\lim_{n\to+\infty}\frac{\log\norm{A^{(n)}(x)}}{n}\\
\chi_{(h,A)}^-(x)=\lim_{n\to+\infty}\frac{\log\conorm{A^{(n)}(x)}}{n}
\end{align*}
provided that these limits exist.

The differential of the endomorphism $f:\TT\to\TT$ defines a cocycle over $f$, and we denote the associated exponents as $\chi_f(x,v), \chi_f^{+}(x), \chi_f^{-}(x)$. It also defines a cocycle $D\hf$ over $\hf$ by
\[
  D_{\hx}\hf=D_{x_0}f,
\]
where $\hx=(x_n)_{n\geq 0}$, which behaves naturally and in particular with the usual notation
\begin{align*}
D_{\hx}\hf^{(n)}=\begin{cases}
D_{\hf^{n-1}\hx}\hf\circ\cdots D_{\hf\hx}\hf\circ D_{\hx}\hf & n>0\\
Id & n=0\\
(D_{\hf^{-n}\hx}\hf)^{-1}\circ\cdots (D_{\hf^{-2}\hx}\hf)^{-1}\circ (D_{\hf^{-1}\hx}\hf)^{-1} & n<0 
\end{cases}
\end{align*}
one gets, by the chain rule, $D_{\hx}\hf^{(n)}=D_{\hx}(\hf^{n})$, i.e.  
\begin{align*}
D_{\hx}\hf^n=\begin{cases}
D_{x_0}f^n & n>0\\
Id & n=0\\
(D_{x_n}f^n)^{-1} & n<0
\end{cases}
\end{align*}

The exponents of this cocycle are denoted as $\chi_{\hf}(\hx,v), \chi_{\hf}^{+}(\hx), \chi_{\hf}^{-}(\hx)$. From the formula above it is evident that
\begin{align*}
\chi_f(\pext(\hx),v)=\chi_{\hf}(\hx,v)\\
\chi_f^{+}(\pext(\hx))=\chi_{\hf}^{+}(\hx)\\
\chi_f^{-}(\pext(\hx))=\chi_{\hf}^{-}(\hx);
\end{align*}
that is, if one quantity is defined the other is defined as well and both coincide. We can now apply Oseledet's theorem (see for example the discussion in \cite{Ruelle1979}) to the cocycle $D\hf$, and use that these coincide almost everywhere with minus the exponents of $D\hf^{-1}$, and hence by the discussion above we can compute the exponents of $f$ using $\hf^{-1}$. Together with the decomposition of $\hmu$ in terms of the $\{\hmu_x\}$ given in Proposition \ref{pro:desintegracionhmu} we get the following facts, that for convenience we record below.

\begin{lemma}\label{lem:exponentes}
There exists a completely $\hf$-invariant set $\hat\reg \subset \hTTf$ of full $\hmu$ measure such that for every $\hx\in\hat\reg, x=\pext(\hx)$, it holds
\begin{enumerate}
  \item $\chi_{\hf}^{+}(\hx)=-\chi_{\hf^{-1}}^{-}(\hx)=\chi^+_f(x)$.
  \item $\chi_{\hf}^{-}(\hx)=-\chi_{\hf^{-1}}^{+}(\hx)=\chi^-_f(x)$.
\end{enumerate}
Moreover, there exists a completely $f$-invariant set $\reg \subset \pext(\hat\reg)$ of full $\mu$ measure so that for every $x\in\reg$, $\hmu_x(\hat\reg)=1$. 

\end{lemma}

In fact, from Oseledet's theorem it also follows that for $\hx \in \hat \reg$ there exist subspaces $E^{-}_{\hx}, E^{+}_{\hx}$ of $\RR$ so that for $v\in E^{\pm}_{\hx}\setminus\{0\}$,
\[
  \chi_{\hf}^{\pm}(\hx)=\chi_{\hf}(\hx,v)=-\chi_{\hf^{-1}}(\hx,v)=-\chi_{\hf^{-1}}^{\mp}(\hx).
\]

\begin{remark}
The bundle $(\pext)_{*}E^{-}$ defines a measurable bundle over $\reg$, since it depends only on the forward orbit of the point. However $E^{+}$ does not project under $\pext$ to a bundle on $\TT$ because $E^{+}$ may be different at different points on the same fiber. 
\end{remark}

\subsection{Relation between Lyapunov exponents and \texorpdfstring{$C_{\chi}(f), C_{\det}(f)$}{}}

The strategy of proof of \hyperlink{theoremA}{Theorem A} is to estimate $\chi_{\hf^{-1}}^{+}$. More precisely, get explicit lower bounds for 
\[\int_{\pi_{ext}^{-1}(x)} \chi_{\hf^{-1}}^{+}\ d\hmu_x\]
at $\mu$-almost every $x \in \Tor$, which by the previous Lemma gives us analogous bounds for $\chi^+_f(x)$ and $\chi^-_f(x)$ $\mu$-almost everywhere.

\smallskip

Recall the definition of $I(x,v;f^n)$ from \eqref{eq:I}. Using the description of the disintegrations of $\hmu_x$ from Proposition \ref{pro:desintegracionhmu}, we obtain that for $x\in\TT$ and $v\in\mathbb S^1\subset\RR$ an unit vector (or equivalently any $(x,v)\in T^1\TT$),
\begin{equation}\label{eq:promedio}
I(x,v;f^n)=\sum_{y\in f^{-n}(x)}\frac{\log\|(D_yf^n)^{-1}v\|}{|\det(D_yf^n)|}=\int_{\pi_{ext}^{-1}(x)} \log\norm{D_{\hy}\hf^{-n}(v)}\ d\hmu_x(\hy).
\end{equation}

Let us fix the following notation. For any unit tangent vector $(x,v)$ at $x\in\TT$ and for any $y\in f^{-i}(x)$, denote $F_y^{-i}v=\frac{(D_yf^i)^{-1}v}{\|(D_yf^i)^{-1}v\|}$, so that $(y,F_y^{-i}v)$ is the unit tangent vector at $y$ obtained by pulling back $v$ under $Df^i$ and normalizing.
\begin{lemma}\label{lem:combinacionconvexa}
For any $n\in\mathbb N$ we have
\[
I(x,v;f^n)=\sum_{i=0}^{n-1}\sum_{y\in f^{-i}(x)}\frac 1{|\det(D_yf^i)|}I(y,F_y^{-i}v;f).
\]
In other words, $\frac 1nI(x,v;f^n)$ is a convex combination of other $I(y,w;f)$.
\end{lemma}

\begin{proof}
We will argue by induction. Assume that the formula holds for $n$, then
\begin{align*}
I(x,v;f^{n+1})&=\sum_{y\in f^{-n-1}(x)}\frac{\log\|(D_yf^{n+1})^{-1}v\|}{|\det(D_yf^{n+1})|}\\
&=\sum_{y\in f^{-n}(x)}\sum_{z\in f^{-1}(y)}\frac{\log\|(D_zf)^{-1}(D_yf^n)^{-1}v\|}{|\det(D_zf)\det(D_yf^n)|}\\
&=\sum_{y\in f^{-n}(x)}\sum_{z\in f^{-1}(y)}\frac{\log\|(D_zf)^{-1}F_y^{-n}v\|+\log\|(D_yf^n)^{-1}v\|}{|\det(D_zf)\det(D_yf^n)|}\\
&=\sum_{y\in f^{-n}(x)}\frac 1{|\det(D_yf^n)|}\sum_{z\in f^{-1}(y)}\frac{\log\|(D_zf^{n})^{-1}F_y^{-n}v\|}{|\det(D_zf)|}\\
&+\sum_{y\in f^{-n}(x)}\frac{\log\|(D_yf^n)^{-1}v\|}{|\det(D_yf^n)|}\sum_{z\in f^{-1}(y)}\frac 1{|\det(D_zf)|}\\
=&\sum_{y\in f^{-n}(x)}\frac 1{|\det(D_yf^n)|}I(y,F_y^{-n}v;f)+\sum_{y\in f^{-n}(x)}\frac{\log\|(D_yf^n)^{-1}v\|}{|\det(D_yf^n)|}\\
&=I(x,v;f^n)+\sum_{y\in f^{-n}(x)}\frac 1{|\det(D_yf^n)|}I(y,F_y^{-n}v;f).
\end{align*}
This concludes the proof of the formula.

Observe that the sum of the coefficients of $\sum_{y\in f^{-i}(x)}\frac 1{|\det(D_yf^i)|}I(y,F_y^{-i}v;f)$ is equal to one for every $i$, so the sum of the coefficients of $I(x,v;f^n)$ is equal to $n$. Then indeed  $\frac 1nI(x,v;f^n)$ is a convex combination of other $I(y,w;f)$.
\end{proof}

\begin{corollary}\label{cor:convex}
For any $(x,v)\in T^1\TT$ and any natural numbers $n,m$, we have that $\frac 1nI(x,v;f^{nm})$ is a convex combination of other $I(y,w;f^m)$.
\end{corollary}

Remember that
\[
C_{\chi}(f)=\sup_{m\in\mathbb N}\frac 1m\inf_{(x,v)\in T^1\TT}I(x,v;f^m).
\]
A straightforward consequence of the definition and of Corollary \ref{cor:convex} is the fact that the definition of $C_{\chi}(f)$ is actually independent of the election of the norm on $T\TT$, and can be rewritten as
\[
C_{\chi}(f)=\lim_{m\rightarrow\infty}\frac 1m\inf_{(x,v)\in T^1\TT}I(x,v;f^m).
\] 

We also have that
\[
  C_{\det}(f)=\sup_{n\in\mathbb N}\frac 1n\inf_{x\in M}\log(|\det(Df_x^n)|)=\lim_{n\rightarrow\infty} \frac 1n \inf_{x\in M}\log(|\det(Df_x^n)|).
\]
\color{black}

The following proposition gives bounds on the Lyapunov exponents at almost every point in terms of $C_{\chi}(f)$ and $C_{\det}(f)$, and establishes the first part of \hyperlink{theoremA}{Theorem A}.

\begin{proposition}\label{pro:existeexpneg}
For $\mu$ almost every $x\in\TT$ we have $\chi^{-}_f(x)\leq -C_{\chi}(f)$ and $\chi^{+}_f(x)\geq C_{\chi}(f)+C_{\det}(f)$.
In particular,
\begin{enumerate}
\item
If $C_{\chi}(f)>0$ ($f\in\mathcal U$), then $\chi_f^-(x)<0<\chi_f^+(x)$ for almost every $x\in\TT$;
\item
If $C_{\chi}(f)>-\frac 12C_{\det}(f)$ ($f\in\mathcal U_1$), then $\chi_f^-(x)<\chi_f^+(x)$ for almost every $x\in\TT$.
\end{enumerate}
\end{proposition}

\begin{proof}
Observe that for $\mu$-a.e. $x\in\TT$ we have
\[
 \chi^{+}_f(x)+\chi^{-}_f(x)=\lim_{n\rightarrow\infty}\frac 1n \log |\det D_xf^n|\geq C_{\det}(f),
\]
thus if $\chi^{-}_f(x)\leq -C_{\chi}$ $\mu$-almost everywhere, this immediately implies that also $\chi^{+}_f(x)\geq C_{\chi}(f)+C_{\det}(f)$.

Let $C_m=\frac 1m\inf_{(x,v)\in T^1\TT}I(x,v;f^m)$. From Corollary \ref{cor:convex} we have that for every $n,m\geq 1$ and unit tangent vector $(x,v)$,
\[
  \frac{I(x,v;g^{nm})}{nm}=\frac{1}{nm}\int_{\pi_{ext}^{-1}(x)} \log\norm{D_{\hy}\hf^{-nm}v}\ d\hmu_x(\hy)\geq C_m.
\]
By Lemma \ref{lem:exponentes} and the dominated convergence theorem it follows that for almost every $x$ and any unit vector $v$,
\[
\int_{\pi_{ext}^{-1}(x)} \chi_{\hf^{-1}}(\hy,v)\ d\hmu_x(\hy)\geq C_m,
\]
therefore
\begin{align*}
-\chi^-_f(x)&=\int_{\pi_{ext}^{-1}(x)} \chi_{\hf^{-1}}^{+}(\hy)\ d\hmu_x(\hy)\geq \int_{\pi_{ext}^{-1}(x)} \chi_{\hf^{-1}}(\hy,v)\ d\hmu_x(\hy)\geq \sup_{m\in\mathbb N}C_m\\
&=C_{\chi}(f).
\end{align*}
\end{proof}

\begin{remark}
Let us make some remarks on the constant $C_{\chi}(f)$ and the NUH property. The fact that we have a negative exponent $\chi^-(\hat x)\leq-\chi$ at $\hat\mu$ almost every point is equivalent to the fact that for $\mu$-almost every $x\in\TT$, and for all unit vectors $v$ with the possible exception of countably many vectors we have
\[
\lim_{n\rightarrow\infty}\frac 1nI(x,v;f^n)\geq\chi.
\]
The possible exceptional vectors come from the possibility of having the following degenerate situation: the $E^+(\hat x)$ spaces coincide for a set of positive $\hat\mu_x$ measure on the fiber $\pi_{ext}^{-1}(x)$.

In other words, the NUH property with uniform bound on the exponents is equivalent to
\[
C_{\chi}'(f)=\essinf_{(x,v)\in T^1\TT}\left(\lim_{n\rightarrow\infty}\frac 1nI(x,v;f^n)\right)\geq\chi>0,
\]
where $C_{\chi}'(f)>C_{\chi}(f)$ would be an optimal bound. We could reformulate the results using $C_{\chi}'(f)$ and a similar $C_{\det}'(f)$ instead of $C_{\chi}(f)$ and $C_{\det}(f)$, however these constants seem to be difficult to estimate, and may not depend semi-continuously on $f$, so the corresponding sets of endomorphisms would not be open. 

Suppose that furthermore $E^+$ is independent of the past: for $\mu$-almost every $x$, for any subspace $E\subset\mathbb R^2$, $\hat\mu_x\left(\{\hat x\in\pi_{ext}^{-1}(x):\ E^+(\hat x)=E\}\right)=0$. Then, for $\mu$-almost every $x\in\TT$ and for {\bf any unit vector $v$} we have
\[
\lim_{n\rightarrow\infty}\frac 1nI(x,v;f^n)\geq\chi.
\]
\end{remark}

\begin{remark}
As we mentioned before, a related approach for establishing positivity of the top Lyapunov exponent (with respect the stationary measure) in the context of random surface diffeomorphisms appears in \cite{Ch20} and \cite{Liu16}. 
\end{remark}

\subsection{Some algebra}\label{ss:Smith}

Denote
\begin{align}
 \mathscr{H}_{NH}(\Tor)=\{E\in\Mat:\det E \neq 0, E\neq k\cdot Id, k\in\Z\}.
\end{align}

 For $E\in \mathscr{H}_{NH}(\Tor)$, let $\tau_1(E)$ be the largest common factor of the entries of $E$; that is,
\begin{align}
 \tau_1(E) = \max\{ k \in \N: E = k\cdot G \text{ for some }G \in \Mat\}.
 \end{align}
Let $\tau_2(E) = |\det(E)|/\tau_1(E)$: then $\tau_2(E)$ is an integer, and since $E=\tau_1(E)\cdot G$ for some integer matrix $G$, $|\det(E)|=\tau_1(E)^2|\det(G)|$ and $\tau_1(E) | \tau_2(E)$. The numbers $\tau_1(E), \tau_2(E)$ are the elementary divisors of $E$. The quantity
\[
\mathrm{d}=| \det E |= \tau_1\cdot \tau_2,
\]
coincides the topological degree of the induced endomorphism $E:\Tor\to\Tor$. We need the following result in linear algebra (see for example \cite{Jacobson2009}).

\begin{theorem}[Smith normal form]
There exist $L, R \in \GL_2(\Z)$ such that $E = L\cdot D\cdot R$, where
\[D = \left(\begin{matrix}
\tau_1 & 0 \\
0 & \tau_2
\end{matrix} \right) .\]
\end{theorem}

For $E:\Tor\to\Tor$, the pre-images of $x\in\TT$ are of the form
\[
E^{-1}(x)=y+\pcov(E^{-1}(\ZZ))
\]
where $E(y)=x$: we are thus interested in understanding the lattice $E^{-1}(\ZZ)$.

\begin{proposition}\label{SNF}
In the hypotheses above, there exists $P \in \GL_2(\Z)$ such that the matrix $G = P^{-1}\cdot E \cdot P$ satisfies:
\begin{equation} \label{preimages}
G^{-1}(\ZZ) = \left\{ \left( \begin{smallmatrix}
\frac{i}{\tau_2} \\
\frac{j}{\tau_1} 
\end{smallmatrix} \right) : i, j \in \Z \right\}.
\end{equation}

Moreover, if $E$ is not a homothety, then $P$ may be chosen so that $e_2=(0,1)$ is not an eigenvector of $G$.

\end{proposition}

\begin{proof}
Let $L,R \in \GL_2(\Z)$ be such that $E = L\cdot D\cdot R$, where $D$ is the Smith Normal form of $E$. Denote by 
\[J = \left( \begin{matrix}
0 & 1 \\
1 & 0 \end{matrix}\right) \]
the involution which permutes $e_1$ and $e_2$ and let $P = R^{-1}J$.  We claim that  the matrix $G = P^{-1}\cdot E\cdot P$ satisfies \eqref{preimages}. Indeed, $G$ can be written as 
\[
G =  J\cdot R\cdot L\cdot D\cdot R\cdot R^{-1}\cdot J = J\cdot R\cdot L\cdot D\cdot J = M\cdot D\cdot J,
\]
where $M = J\cdot R\cdot L\in\GL$. Hence
\[
G^{-1}(\ZZ) = J^{-1}D^{-1}(\ZZ),
\]
which is clearly in the form expressed in \eqref{preimages}.

\smallskip

If we further assume that $E\in \mathscr{H}_{NH}(\Tor)$, then it has at most two eigenvectors in $\ZZ$. In particular, there is some $k \in \Z$ such that $k\cdot e_1+e_2$ is not an eigenvector of $E$. Let 
\[S = \left(\begin{matrix}
1 & k \\
0 & 1
\end{matrix} \right),\]
and take $ P' = P\cdot S$. It follows that $e_2$ is not an eigenvector of 
\[
G' = {P'}^{-1}\cdot E \cdot P' = S^{-1}\cdot G\cdot S.
\]

Moreover, since $\tau_1$ divides $\tau_2$, $S$ acts as a permutation on $G^{-1}(\ZZ)$. Hence \eqref{preimages} holds with $G'$ in place of $G$.
\end{proof}

%%%%%%%%%%%%%%%%%%%%%%%%%%%%%%%%%%%%%%%%%%%%%%%%%%%%%%%%%%%%%%%%%%%%%%%%%%%%%%%%%%%%%%%%%%%%%%%%%%%%%%%%%%%%%%%%%%%%%%%
\section{Construction of non-uniformly hyperbolic endomorphisms}\label{sec:shears}
%%%%%%%%%%%%%%%%%%%%%%%%%%%%%%%%%%%%%%%%%%%%%%%%%%%%%%%%%%%%%%%%%%%%%%%%%%%%%%%%%%%%%%%%%%%%%%%%%%%%%%%%%%%%%%%%%%%%%%%
%%%%%%%%%%%%%%%%%%%%%%%%%%%%%%%%%%%%%%%%%%%%%%%%%%%%%%%%%%%%%%%%%%%%%%%%%%%%%%%%%%%%%%%%%%%%%%%%%%%%%%%%%%%%%%%%%%%%%%%
%%%%%%%%%%%%%%%%%%%%%%%%%%%%%%%%%%%%%%%%%%%%%%%%%%%%%%%%%%%%%%%%%%%%%%%%%%%%%%%%%%%%%%%%%%%%%%%%%%%%%%%%%%%%%%%%%%%%%%%

In this section we will show that $\mathcal U$ has elements in almost all the homotopy classes of endomorphisms.

\subsection{Endomorphisms} % (fold)
\label{sub:endomorphisms}

Fix $E\in \mathscr{H}_{NH}(\Tor)$ with elementary divisors $\tau_1, \tau_2$, and denote by $\mathrm{d}=\tau_1\cdot \tau_2$ its degree. We assume that $\tau_2\geq 5$. By making the linear change of coordinates indicated in Proposition \ref{SNF} we can assume that $\Pp E$ does not fix $[(0,1)]$ and $E^{-1}\mathbb Z^2=\frac 1{\tau_2}\mathbb Z\times\frac 1{\tau_1}\mathbb Z$.

Let $\delta<\frac 1{4\tau_2}$. Choose points $x_1,x_2,x_3,x_4\in\mathbb T^1$, in this order, such that
\begin{itemize}
\item $I_1=[x_1,x_2]$ and $I_3=[x_3,x_4]$ have size $2\delta$;
\item the translation of $I_1$ by a multiple of $\frac 1{\tau_2}$ does not intersect $I_3$;
\item $I_2=(x_2,x_3)$ and $I_4=(x_4,x_1)$ have size strictly larger than $\frac 1{\tau_2}  \left[\frac{\tau_2-1}2\right]\footnote{$[a]$ denotes the integer part of the real number $a$.}$.
\end{itemize}

\begin{remark}\label{rem:odd}
If $\tau_2$ is odd then the points can be $\{\frac 14-\delta,\frac 14+\delta, \frac 34-\delta,\frac 34+\delta\}$.\\
\end{remark}
\begin{definition}\hfill
\begin{enumerate}
	\item The critical region is $\mathcal{C}=(I_1\cup I_3)\times \mathbb T^1\subset \Tor$, and its complement $\mathcal G$ is the good region.
\item
	The good region is divided into $\mathcal G^+=I_4\times\mathbb T^1$ and $\mathcal G^-=I_2\times\mathbb T^1$.
\end{enumerate}	
\end{definition}

\begin{remark}
Observe that $\mathcal C$ is closed and $\mathcal G^+$ and $\mathcal G^-$ are open.
\end{remark}

\begin{lemma}\label{lem:preimagenesd}
	For every $x\in\Tor$, $E^{-1}(x)$ has $\mathrm{d}$ points. At least $\tau_1\left[\frac{\tau_2-1}2\right]$ pre-images are inside each one of $\mathcal G^+$ and $\mathcal G^-$, and at most $\tau_1$ of them are inside $\mathcal C$.

Furthermore, at least one pre-image of $x$ is inside $\mathcal G$ and $d(x,\mathcal C)>\frac 1{10}$.
\end{lemma}

\begin{proof}
    The first part follows from Proposition \ref{SNF} and the construction of the regions.  The last statement is consequence of the fact that at least one good region has size greater than $\frac{1-4\delta}2>\frac {\tau_2-1}{2\tau_2}$, and if we take a band of size $\frac 1{\tau_2}$ in the middle of this good region, at least one pre-image has to fall in the middle band. Then the distance to the boundary is greater than $\frac 12\left(\frac{\tau_2-1}{2\tau_2}-\frac 1{\tau_2}\right)=\frac{\tau_2-3}{4\tau_2}\geq \frac 1{10}$ because $\tau_2\geq 5$.
\end{proof}

\subsection{Shears}
\label{sub:shears}

We start studying the dynamics of a family of shears that we will use to deform the original endomorphism. Let $s:\mathbb T^1\rightarrow\mathbb R$. Consider the family of conservative diffeomorphisms of the torus
\begin{equation}
	h_t(x_1,x_2)=(x_1, x_2+ts(x_1)).
\end{equation}

\begin{remark}\label{rem:standard}
If $U:\Tor\to\Tor$ is the Dehn twist given by $U(x_1,x_2)=(x_1+x_2,x_1)$, and $s(x_1)=\sin(2\pi x_1)$, then $\{s_t=U\circ h_t\}_{t>0}$ is the famous \emph{standard family} \cite{Chirikov2008} and has the formula $s_t(x_1,x_2)=(x_1+x_2+t\sin(2\pi x_1), x_2+t\sin(2\pi x_1))$.
\end{remark}

We observe that
\begin{equation}
	D_{(x_1,x_2)}h_t=\begin{pmatrix}
	1 & 0\\
	ts'(x_1) & 1
	\end{pmatrix}.
\end{equation}

As we mentioned before, $C_{\chi}(f)$ is independent of the choice of the norm on the tangent bundle of $\TT$. In order to simplify the computations, we will choose the maximum norm:
$$
\|(v_1,v_2)\|=\max\{|v_1|,|v_2|\}.
$$

All the computations in this section are performed using this norm. We will try to keep track of the constants involved and this will make the exposition a bit more technical, however it will help us give some concrete examples in subsection \ref{sub:concrete_examples}.

A first observation is the following.

\begin{lemma}
\[
\left\|\begin{pmatrix}
1 & 0\\
a & 1\\
\end{pmatrix}\right\|=\frac 1{\conorm{\begin{pmatrix}
1 & 0\\
a & 1\\
\end{pmatrix}}}=a+1
\]
\end{lemma}
\color{black}

Let $0<a<b$. There exists an analytic map $s$ which satisfies the following conditions: 
\begin{enumerate}
\item
If $u\in I_4$ then $a<s'(u)<b$;
\item
If $u\in I_2$ then $-b<s'(u)<-a$;
\item
If $u\in I_1\cup I_3$ then $|s'(u)|<b$.
\end{enumerate}

\begin{remark}
If $\tau_2$ is odd, in view of the Remark \ref{rem:odd} we can take $s(u)=\sin(2\pi u)$, with $a=2\pi\sin(2\pi\delta)-\epsilon$ and $b=2\pi+\epsilon$, for any $\epsilon>0$.
\end{remark} 

As a consequence we get, for every $x\in \Tor$
\begin{align}\label{eq:normaht}
	 \|D_xh_t\|=\frac 1{\conorm{D_xh_t}}<bt+1
\end{align}

 \begin{definition}
Given any $\alpha>0$, the corresponding horizontal cone is $\Deh=\{(u_1,u_2)\in\Real^2: |u_2|\leq \alpha|u_1|\}$, while the corresponding vertical cone is its complement $\Dev=\Real^2\setminus\Deh$.
\end{definition}

Since $E$ does not fix the vertical vector $(0,1)$, there exists $\alpha>1$ such that $E^{-1}\Dev\subset\overline{E^{-1}\Dev}\subset Int(\Deh)\subset\Deh$, and from now on we fix such an $\alpha$. Observe that $\Dev$ is open and $\Deh$ is closed.

\begin{remark}\label{rem:twist}
For a Dehn twist (eventually multiplied with a homothety) we have $\alpha=2+\epsilon$.
\end{remark}

\begin{lemma}\label{lem:estimativasht}
For every $t>\frac{2\alpha}a$ and for every unit vector $v\in T_x\Tor$ the following holds.
\begin{enumerate}
		\item If $v\in \Deh$ and
		\begin{enumerate}
			\item $x\in\mathcal G$, then
			\begin{itemize}
				\item \((D_xh_t)^{-1}v\in\Dev\) ($D_xh_t^{-1}\Deh\subset\Dev$)
				\item \(\norm{(D_xh_t)^{-1}v}> \frac{at-\alpha}{\alpha}=t\frac{a-\frac{\alpha}t}{\alpha};\)
			\end{itemize}
			\item $x\in\mathcal C$, then
			\begin{itemize}
				\item \(\norm{(D_xh_t)^{-1}v}> \frac{1}{\alpha}\).	
			\end{itemize}
		\end{enumerate}
		\item If $v=\pm(v_1,1)\in \Dev$ then 
		\begin{enumerate}
			\item either for every $x\in \mathcal G^+$ (if $v_1\leq 0$) or for every $x\in \mathcal G^-$ (if $v_1\geq 0$) it holds:
			\begin{itemize}
				\item \((D_xh_t)^{-1}v\in\Delta_{\alpha}^v\),
				\item \(\norm{(D_xh_t)^{-1}v}> 1;\)
			\end{itemize}
			\item for all the other $x$, we have 
			\begin{itemize}
				\item \(\norm{(D_xh_t)^{-1}v}> \frac{1}{bt+1}\).	
			\end{itemize}
		\end{enumerate}			
\end{enumerate}
\end{lemma}

\begin{proof}
	For the boundary vectors $v=(1,\alpha), w=(1,-\alpha)$ we compute
	\begin{align*}
		(D_xh_t)^{-1}v=(1,-ts'(x_1) +\alpha)\quad  (D_xh_t)^{-1}w=(1,-ts'(x_1)( -\alpha).
	\end{align*}
	It follows that if $x\in\mathcal{G}$ then
	\[
	|-ts'(x_1) +\alpha|\geq ta-\alpha>\alpha
	\]
	if $t>\frac{2\alpha}a$. Likewise for $w$, which shows that $(D_xh_t)^{-1}\Deh\subset \Dev$.
	
	We also remark that the minimal expansion of $D_xh_t^{-1}$ on $\Deh$ for $x\in\mathcal G$ is realized on $v$ or on $w$. Then for every $v\in \Deh$ and for every $x\in \mathcal{G}$ we have
	\[
	\norm{(D_xh_t)^{-1}v}> \frac{\|(1,ta-\alpha)\|}{\|(1,\alpha)\|}\geq\frac{ta-\alpha}{\alpha}.
	\]   
	Part (b) is similar and left as an exercise.
	
	The last part follows by noting that $s'$ has constant sign in $\mathcal{G}^+, \mathcal{G}^-$, and
	\[
	\conorm{(D_xh_t)^{-1}}=\frac 1{\|D_xh_t\|}> \frac{1}{bt+1}, \forall x\in \Tor.
	\]  
\end{proof}

\subsection{Endomorphisms and shears} % (fold)
\label{sub:endomorphismsandshears}

We now consider the analytic maps
\[
f_t=E\circ h_t;	
\]
clearly $f=f_t$ is an area preserving endomorphism isotopic to $E$.

The regions $\mathcal C$, $\mathcal G^+$ and $\mathcal G^-$ are invariant under $h_t$, and furthermore $\mathcal C$ is closed while $\mathcal G^+$ and $\mathcal G^-$ are open. A direct consequence of Lemma \ref{lem:preimagenesd} is the following:

\begin{lemma}\label{lem:preimagenesdd}

For every $x\in\Tor$, $f^{-1}(x)$ has $\mathtt{d}$ points. At least $\tau_1[\frac{(\tau_2-1)}2]$ pre-images are inside each one of $\mathcal G^+$ and $\mathcal G^-$, and at most $\tau_1$ of them are inside $\mathcal C$.

Furthermore, at least one pre-image of $x$ is inside $\mathcal G$ and $d(x,\mathcal C)>\frac 1{10}$.
\end{lemma}

Let $e_v=\inf_{v\in\Dev,\ \|v\|=1}\|E^{-1}v\|$, $e_h=\inf_{v\in\Deh,\ \|v\|=1}\|E^{-1}v\|$. From Lemma \ref{lem:estimativasht} we get:

\begin{lemma}\label{lem:estimativasst} 
For $t>\frac{2\alpha}a$ it holds
\begin{enumerate}
	\item $x\in \mathcal G$ then $(D_xf)^{-1}\Dev\subset\overline{(D_xf)^{-1}\Dev}\subset\Dev$ (it is strictly invariant).
		\item if $v\in \Dev$ is a unit vector, then
	\begin{align*}
	\|(D_xf)^{-1}v\|>\begin{dcases}
	\frac{e_v(a-\frac{\alpha} t)}{\alpha} t & x\in \mathcal{G}\\
	\frac {e_v}{\alpha} & x\in \mathcal{C}
	\end{dcases}
	\end{align*}
	\item  If $v\in\Deh$, and $E^{-1}v=(v_1,v_2)$, let $*(v)$ be the sign of $-\frac {v_1}{v_2}$ (0 and $\pm\infty$ have both + and - sign). Then for all $x\in \mathcal G^{*(v)}$ we have $(D_xf)^{-1}(v)\in\Dev$.
	\item If $v\in\Deh$ is a unit vector, then
	\begin{align*}
	\|(D_xf)^{-1}v\|>\begin{dcases}
	e_h & x\in \mathcal G^{*(v)}\\
	\frac{e_h}{b+\frac 1t}t^{-1} & x\notin \mathcal G^{*(v)}
	\end{dcases}
	\end{align*}
\end{enumerate}
\end{lemma}

\subsection{Non-uniform hyperbolicity} % (fold)
\label{sub:non_uniform_hyperbolicity}

In this part we will establish the remaining part of \hyperlink{theoremA}{Theorem A} and prove that for $t$ sufficiently large, the map $f=f_t$ is NUH. The arguments are inspired in the abstract methods used in \cite{extendedflex} (see also \cite{RandomProdSt}), but with the additional difficulty of having to work in the natural extension instead of directly in the manifold. We rely on Proposition \ref{pro:existeexpneg}.

\begin{proof}[Proof of Theorem A ($\mathcal U$ intersects the isotopy class of $E$)]

For any $(x,v)\in T\Tor, v\neq 0$ and $n\geq 0$ denote by $Df^{-n}(x,v)$ all the preimages of $(x,v)$ under $Df$:
$$Df^{-n}(x,v)=\{(y,w)\in T\TT:\ f^n(y)=x,\ D_yf^nw=v\}.$$
For any nonzero tangent vector $(x,v)$, define
\begin{align}
&\mathcal{G}_n=\{(z,w)\in Df^{-n}(x,v):\ w\in\Dev\}\\
&\mathcal{B}_n=Df^{-n}(x,v)\setminus\mathcal G_n\\
&g_n=\#\mathcal{G}_n\\
&b_n=\#\mathcal{B}_n=\mathtt d^{n}-g_n.
\end{align}
\color{black}

From lemmas \ref{lem:preimagenesdd},\ref{lem:estimativasst} one deduces:

\begin{lemma}\label{lem:good}
Let $(x,v)\in T\Tor$.
\begin{enumerate}
	\item If $v\in\Dev$  then at least $\tau_1(\tau_2-1)$ of its pre-images under $Df$ are also in $\Dev$.
	\item If $v\in\Deh$  then at least $\tau_1\left[\frac{\tau_2-1}2\right]$ of its pre-images under $Df$ are in $\Dev$.
\end{enumerate}
\end{lemma}

By the lemma above one can compute 
\begin{align*}
g_{n+1}&\geq \tau_1(\tau_2-1)\cdot g_n+\tau_1\left[\frac{\tau_2-1}{2}\right]\cdot b_n= \tau_1(\tau_2-1)\cdot g_n\\
&+\tau_1\left[\frac{\tau_2-1}{2}\right]\cdot(\mathrm d^n-g_n)\\
&=\tau_1\left(\tau_2-1-\left[ \frac{\tau_2-1}{2}\right]\right)g_n+\tau_1\left[\frac{\tau_2-1}{2}\right]\cdot\mathrm d^n\\
&\Rightarrow \frac{g_{n+1}}{\mathrm d^{n+1}}\geq \frac 1\tau_2\left(\tau_2-1-\left[ \frac{\tau_2-1}{2}\right]\right)\frac{g_n}{d^n}+\frac 1{\tau_2}\left[\frac{\tau_2-1}{2}\right]
\end{align*}
Let us denote $a_n=\frac{g_n}{\mathrm d^{n}}$  and 
\begin{align*}
c&=\frac 1\tau_2\left(\tau_2-1-\left[ \frac{\tau_2-1}{2}\right]\right)\\
e&=\frac 1{\tau_2}\left[\frac{\tau_2-1}{2}\right]
\end{align*}
The inequality above becomes 
\[
	a_{n+1}\geq c\cdot a_n+e,
\]
Recalling that we are assuming $\tau_2\geq 5$ we get:

\begin{lemma}\label{lem:proporcion}
For every $(x,v)\in T\Tor, v\neq 0$ and $n\geq 0$ it holds
\[
	a_n\geq \frac{\left[\frac{\tau_2-1}2\right]}{1+\left[\frac{\tau_2-1}2\right]}(1-c^n)
\]
In particular, 
\[
\liminf_n a_n\geq \frac{\left[\frac{\tau_2-1}2\right]}{1+\left[\frac{\tau_2-1}2\right]}\geq\frac{2}{3},
\]
uniformly in $(x,v)$.
\end{lemma} 

\begin{proof}
By induction we obtain
\[
a_{n}\geq e\sum_{k=0}^{n-1}c^k+c^{n}a_0
\]
and therefore 
\begin{align*}
a_n&\geq\frac {e(1-c^n)}{1-c}=\frac{\left[\frac{\tau_2-1}2\right]}{1+\left[\frac{\tau_2-1}2\right]}(1-c^n).
\end{align*}
\end{proof}

Recall the definition of $I(x,v;f)$ from \eqref{eq:promedio}. We have the following direct consequence of Lemma \ref{lem:estimativasst}.
\begin{corollary}\label{cor:estimativasst}
If $t\geq\frac{2\alpha}a$ then
\begin{enumerate}
	\item If $v\in\Dev$ then
	\[
	I(x,v;f)\geq \left(1-\frac{1}{\tau_2}\right)\log t+\log \frac{e_v}{\alpha} \left(a-\frac{\alpha}t\right)^{1-\frac 1{\tau_2}}
	\]
	\item If $v\in\Deh$ then
	\[
	I(x,v;f)\geq  -\left(1-\frac1{\tau_2}\left[\frac{\tau_2-1}2\right]\right)\log t+\log e_h\left(b+\frac 1t\right)^{-\left(1-\frac 1{\tau_2}\left[\frac{\tau_2-1}2\right]\right)}
	\]
\end{enumerate}
\end{corollary}

Now we will use Lemma \ref{lem:combinacionconvexa}:
\[
I(x,v;f^n)=\sum_{i=0}^{n-1}\sum_{y\in f^{-i}(x)}\frac 1{\det(D_yf^i)}I(y,F_y^{-i}v;f):=\sum_{i=0}^{n-1}J_i,
\]
where
\[
J_i=\sum_{y\in f^{-i}(x)}\frac 1{\det(D_yf^i)}I(y,F_y^{-i}v;f).
\]

Since the determinant of $Df$ is constant and equal to $\mathtt d$ we obtain
\[
	J_i=\frac{1}{\mathtt d^i}\sum_{y\in f^{-i}(x)}I(y,F_y^{-i}v;f)=\frac{1}{\mathtt d^i}\sum_{(y,w)\in\mathcal{G}_i} I(y,w;f)+\frac{1}{\mathtt d^i}\sum_{(y,w)\in\mathcal{B}_i} I(y,w;f)
\]

 It follows, for $t\geq \frac{2\alpha}{a}$,
 \begin{align*}
\lim_{i\rightarrow\infty}J_i\geq&\frac{\left[\frac{\tau_2-1}2\right]}{1+\left[\frac{\tau_2-1}2\right]}\left(\left(1-\frac{1}{\tau_2}\right)\log t+\log \frac{e_v}{\alpha} \left(a-\frac{\alpha}t\right)^{1-\frac 1{\tau_2}}\right)+\\
&\frac 1{1+\left[\frac{\tau_2-1}2\right]}\Bigg( -\left(1-\frac1{\tau_2}\left[\frac{\tau_2-1}2\right]\right)\log t\\
&+\log e_h\left(b+\frac 1t\right)^{-\left(1-\frac 1{\tau_2}\left[\frac{\tau_2-1}2\right]\right)}\Bigg)=\frac{\left[\frac{\tau_2-1}2\right]-1}{\left[\frac{\tau_2-1}2\right]+1}\log t+C(t),
\end{align*}
Where
$$
C(t)=\log \left(\frac{e_v}{\alpha}\right)^{\frac{\left[\frac{\tau_2-1}2\right]}{1+\left[\frac{\tau_2-1}2\right]}}e_h^{\frac 1{1+\left[\frac{\tau_2-1}2\right]}}\left(a-\frac{\alpha}t \right)^ {\frac{(\tau_2-1)\left[\frac{\tau_2-1}2\right]}{\tau_2\left(1+\left[\frac{\tau_2-1}2\right]\right)}}\left(b+\frac 1t\right)^{\frac{-\tau_2+\left[\frac{\tau_2-1}2\right]}{\tau_2\left(1+\left[\frac{\tau_2-1}2\right]\right)}}
$$
is uniformly bounded from below by some constant $C$. Since $\tau_2(E)\geq 5$ it holds
$$
\frac{\left[\frac{\tau_2-1}2\right]-1}{\left[\frac{\tau_2-1}2\right]+1}\geq\frac 13>0.
$$
Since all the bounds above are uniform for all nonzero tangent vectors $(x,v)$, we obtain that for $t$ sufficiently large, for all $i$ greater than some $i_0$, and for all nonzero tangent vectors $(x,v)$ we have $J_i(x,v)>k>0$. This implies that there exists some $n_0$ such that
$$
\frac 1{n_0}I(x,v;f^{n_0})=\frac 1{n_0}\sum_{i=0}^{n_0-1}J_i(x,v)>\frac k2>0,
$$
for all nonzero tangent vectors $(x,v)$. This implies that $C_{\chi}(f)>0$, and Proposition \ref{pro:existeexpneg} implies that $f$ is NUH. This finishes the proof of \hyperlink{theorem A}{Theorem A}. \label{pag:finalprueba}

\end{proof}

\begin{remark}
Observe that the asymptotic bound on the Lyapunov exponent is
$$
\lambda_f^-(x)\leq\approx-\frac{\left[\frac{\tau_2-1}2\right]-1}{\left[\frac{\tau_2-1}2\right]+1}\log t,
$$
which approaches $-\log t$ when $\tau_2$ is large.
\end{remark}

\subsection{A concrete example} % (fold)
\label{sub:concrete_examples}

As mentioned in remark \ref{rem:hipotesisteoA} our methods provide effective bounds for the $\mathcal C^1$ distance between $f=f_t$ and $E$, which amounts to estimate $t$ in terms of $E$. This is direct from the arguments of the previous part. Of course, without any information about the linear map there is no possibility to obtain a concrete number, so here we present an example of this computation mostly to convince the reader that bounds obtained are manageable.

Let us start by considering the linear expanding map $E_k:\Tor\to\Tor$ induced by the matrix 
\[
	\begin{pmatrix}
	2k+1 & 2k+1\\
	0 & 2k+1
	\end{pmatrix},
\]
and the corresponding deformation $f_t=E_k\circ h_t$. Then $E_k$ has degree $(2k+1)^2$ and does not fix the vertical direction: in fact $E_k(0,1)=(2k+1)(1,1)$. \color{black}From this we see that we can take $\alpha=2$. We also have $e_v=\frac 1{2k+1}$ and $e_h=\frac 1{2(2k+1)}$

Note that $\tau_1=\tau_2=2k+1$, so $\left[\frac{\tau_2-1}2\right]=k$. Furthermore
\begin{align*}
&\frac{\left[\frac{\tau_2-1}2\right]-1}{\left[\frac{\tau_2-1}2\right]+1}=\frac{k-1}{k+1},\\
&\left(\frac{e_v}{\alpha}\right)^{\frac{\left[\frac{\tau_2-1}2\right]}{1+\left[\frac{\tau_2-1}2\right]}}e_h^{\frac 1{1+\left[\frac{\tau_2-1}2\right]}}=\frac 1{2(2k+1)}.
\end{align*}
For $\delta=\frac 1{4(2k+1)}$ we have $b=2\pi$ and $a=2\pi\sin\frac{\pi}{2(2k+1)}$.

Then 
\[
C(t)=\log\frac 1{2(2k+1)}\left(2\pi\sin\frac{\pi}{2(2k+1)}-\frac 2t\right)^{\frac{2k^2}{(k+1)(2k+1)}}\left(2\pi+\frac 1t\right))^{-\frac 1{2k+1}}.
\]
In order to obtain NUH we need
\[
t>\left(2(2k+1)\right)^{\frac{k+1}{k-1}}\left(2\pi\sin\frac{\pi}{2(2k+1)}-\frac 2t\right)^{-\frac{2k^2}{(k-1)(2k+1)}}\left(2\pi+\frac 1t\right)^{\frac {k+1}{(k-1)(2k+1).}}
\]

\begin{itemize}
	\item For $k=2$ (multiplication by $5$) we can take $t=1042$.
	\item For $k=3$ (multiplication by $7$) we can take $t=216$.
    \item For $k=5$ (multiplication by $11$) we can take $t=151$.
\end{itemize}

On the other hand, in this case $C_{\det}(f)=2\log (2k+1)$, so in order to have $f=f_t\in\mathcal U_1$ and obtain continuity of the exponents the condition on $t$ is much better. We need
\[
t>2^{\frac{k+1}{k-1}}\left(2\pi\sin\frac{\pi}{2(2k+1)}-\frac 2t\right)^{-\frac{2k^2}{(k-1)(2k+1)}}\left(2\pi+\frac 1t\right)^{\frac {k+1}{(k-1)(2k+1).}}
\]

\begin{itemize}
	\item For $k=2$ (multiplication by $5$) we can take $t=10.02$.
    \item For $k=3$ (multiplication by $7$) we can take $t=6.29$.
\end{itemize}

We definitely do not claim that this bounds are optimal, and we are sure that they can be considerably improved. Our estimates do not take into consideration a better description of the distribution inside the torus of preimages of higher order, and better bounds for the expansion of preimages of higher order of tangent vectors.

\smallskip

Observe the following interesting consequence: we can write
\[
	f_t=E_2\circ h_t= 5s_t
\]
where $s_t:\Tor\to\Tor$ is the standard family, and in particular $Ds_t$ defines a cocycle over $5s_t$. Denoting $S_t=(5s_t,Ds_t)$ it follows that the exponents of $f_t$ and $S_t$ are related by
\[
	\chi_{f_t}^{\pm}=\chi^{\pm}_{S_t}+\log 5
\] 

\begin{corollary}[``Expanding'' Standard Map]
For $t\geq 1043$ the cocycle $S_t$ is non-uniformly hyperbolic.
\end{corollary}

Proving the existence of parameters for which $s_t$ is non-uniformly hyperbolic is one of the most important problems in smooth ergodic theory. Some related results (for skew-products) are: \cite{NUHD}, \cite{LyaRandom}.

\begin{figure}[H]
	%\centering
	\subfloat[]{
		\includegraphics[width=0.48\textwidth]{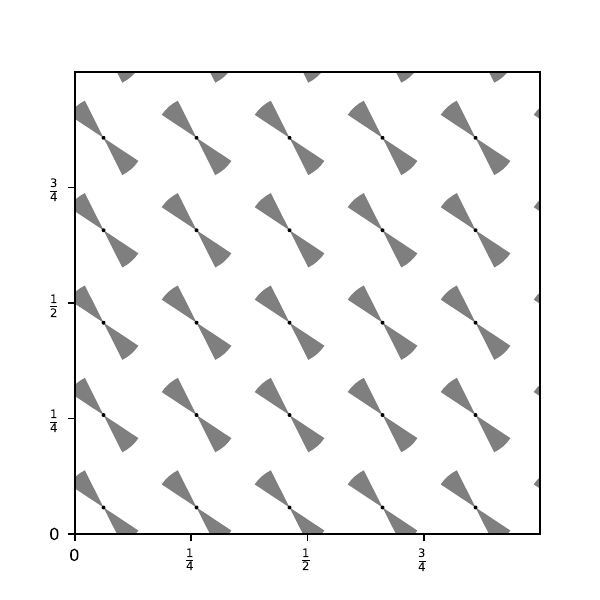}}
		\qquad
	\subfloat[]{
		\includegraphics[width=0.48\textwidth]{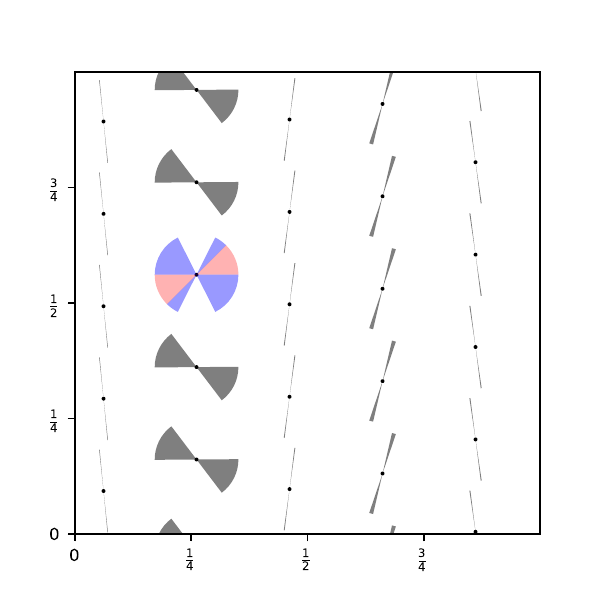}}
		\qquad
	\subfloat[]{
		\includegraphics[width=0.48\textwidth]{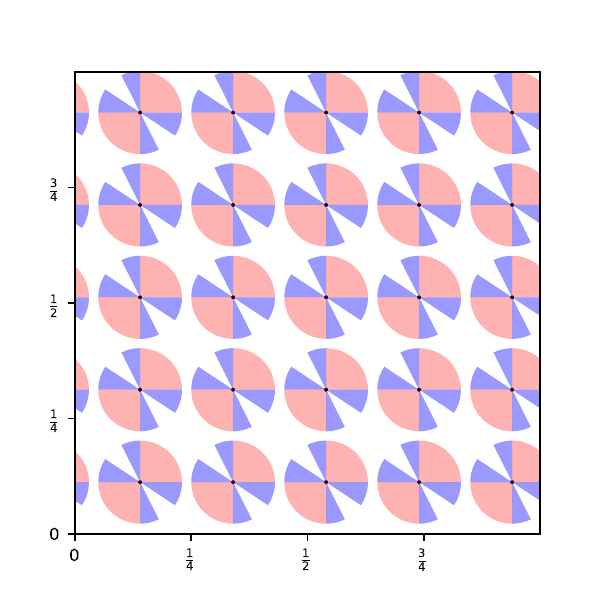}}
		\qquad
	\subfloat[]{
		\includegraphics[width=0.48\textwidth]{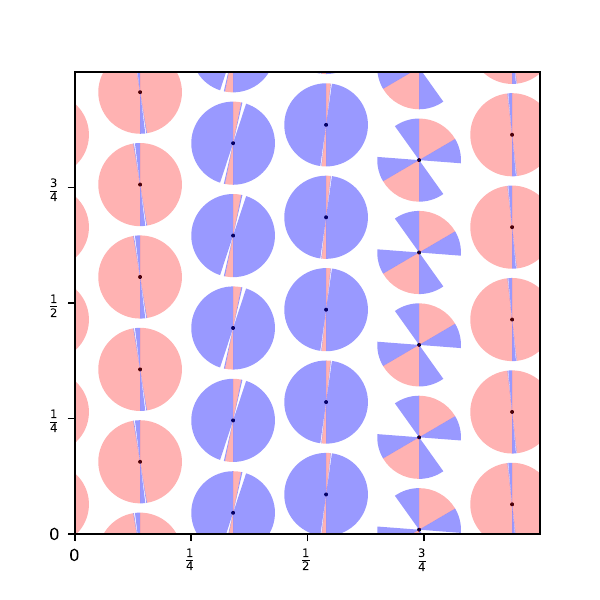}}
	\caption{Figure (a) shows the $25$ preimages of a randomly chosen point
	$x = (0.594, 0.287)$ under the map $E_k$ when $k=2$, together with the
cones $E_k^{-1} \Delta_\alpha^v$ with $\alpha = 2$. Figure (b) shows the
preimages of these points under the map $S_t^{-1} = S_{-t}$ for $t=1.5$,
together with the corresponding cones. (Taking $t$ bigger would make some of the
cones hard to see because they would be too thin.) In this example we see that
$(Df_{y})^{-1} \Delta_\alpha^v \subset \Delta_\alpha^v$ for $20$ choices of $y
\in f^{-1}(x)$. At one of the $y$'s (say $y_0$) for which this fails, the
illustration shows the horizontal cone $\Delta_\alpha^h$ instead of
$(Df_{y_0})^{-1}\Delta_\alpha^v$. The horizontal cone at $y_0$ is divided into
two parts, reflecting the quadrant to which it is mapped under $E^{-1}$. 
Finally, in figure (d) we see $(Df_z)^{-1}\Delta_\alpha^h$ for every $z \in f^{-1}(y_0)$. 
For $10$ choices of $z$, vectors in $\Delta_\alpha^h$ which lie in the second
and fourth quadrant get mapped inside $\Delta_\alpha^v$ 
while, for $10$ other choices, vecotrs in the first and third quadrant get mapped into $\Delta_\alpha^v$.} 		
\end{figure}

%%%%%%%%%%%%%%%%%%%%%%%%%%%%%%%%%%%%%%%%%%%%%%%%%%%%%%%%%%%%%%%%%%%%%%%%%%%%%%%%%%%%%%%%%%%%%%%%%%%%%%%%%%%%%%%%%%
\section{Continuity of the characteristic exponents}
\label{sec:continuity}
%%%%%%%%%%%%%%%%%%%%%%%%%%%%%%%%%%%%%%%%%%%%%%%%%%%%%%%%%%%%%%%%%%%%%%%%%%%%%%%%%%%%%%%%%%%%%%%%%%%%%%%%%%%%%%%%%%%
%%%%%%%%%%%%%%%%%%%%%%%%%%%%%%%%%%%%%%%%%%%%%%%%%%%%%%%%%%%%%%%%%%%%%%%%%%%%%%%%%%%%%%%%%%%%%%%%%%%%%%%%%%%%%%%%%%%    

In this part we prove \hyperlink{theoremB}{Theorem B}. Recall that from Proposition \ref{pro:existeexpneg} we already know that if $f\in\mathcal U_1$, then for $\mu$ almost every point $x\in\TT$ (or $\hat\mu$-a.e. $\hat x\in \hTTf$) we have that the two Lyapunov exponents at $x$ (resp. $\hat x$) are different.

Let us spell some facts used in the proof. The first lemma is well known in the theory (cf.\@ \cite{LyaViana}).

\begin{lemma}[Projectivized cocycles]\label{le:projective}
Assume that $A$ is a two-dimensional cocycle over the invertible map $f:M\rightarrow M$ with invariant measure $\mu$. Let $\Pp f:\Pp M\rightarrow \Pp M$ be its projectivization, which is a bundle map over $f$. Assume that the two Lyapunov exponents $\lambda^+(x)$ and $\lambda^-(x)$ are different at $\mu$-almost every point, with the corresponding Lyapunov subspaces $E^+(x)$ and $E^-(x)$. Then:
\begin{itemize}
\item
If $\mu$ is ergodic then there are exactly two ergodic lifts of $\mu$ to $PM$, $\mu^{p+}$ and $\mu^{p-}$. The disintegrations of $\mu^{p+}$ along the fibers of $\Pp M$ are exactly the Dirac measures at the Lyapunov spaces $E^+$, and a similar statement holds for $\mu^{p-}$.
\item
Suppose that $\mu^p$ is a lift of $\mu$ (not necessarily ergodic) to $\Pp M$, then there is a measurable $f$-invariant function $\rho:M\rightarrow[0,1]$ such that the disintegrations of $\mu^p$ along the fibers of $\Pp M$ are $\rho(x)\delta_{E^+(x)}+(1-\rho(x))\delta_{E^-(x)}$.
\end{itemize}
\end{lemma}

\subsection{Product measures} We will be concerned with measures which have product disintegrations. The following has a direct analogue in previous work for the so-called ``cocycles with continuous holonomy'', and ``$s$-states and $u$-states''. We state the results in the general case of products of metric spaces: \color{black} the existence of holonomy translates in our setting to the product structure of the spaces involved, while the invariance of disintegrations of measures under holonomies translates to the product structure of the corresponding measures.

The next result is a characterization of product measures. Let us remark that what we call ``product measure'' is stronger than the notion of ``measure with product structure'' which is common in the literature, more precisely we ask that the density is constant, so the measure on the product space is exactly the product of two measures on the two factor spaces. \color{black}

\begin{lemma}[Product measures]\label{le:productmeasures}
Let $X$, $Y$ be compact metric spaces and $\mu$ a Borel probability measure on $X\times Y$. Then $\mu$ is a product measure (i.e. $\mu=(\pi_{X})_*\mu \times 
(\pi_{Y})_* \mu$) if and only if, for any $f\in C(X,\mathbb R)$ and $g\in C(Y,\mathbb R)$ we have
\begin{align}\label{eq:product}
\nonumber\int_{X\times Y}f\circ\pi_X \cdot g\circ \pi_Y \ d\mu &=
\int_{X\times Y} f \circ \pi_X \ d\mu\int_{X\times Y } g \circ \pi_Y \ d\mu\\
& = \int_X f\ d(\pi_{X})_* \mu \cdot \int_Y g \ d(\pi_{Y})_* \mu.
\end{align}
\end{lemma}

\begin{proof}
The condition on the integrals is equivalent to the fact that any pair of measurable functions $f:(X,(\pi_{X})_*\mu) \to \Real$ and $g:(Y,(\pi_{Y})_* \mu) \to \Real$ are independent, and this is equivalent to $\mu$ being a product measure.

\end{proof}

The next result says that the property of being a product measure is closed under weak* limits.

\begin{corollary}[Limits of product measures]
A weak* limit of Borel product probability measures is a Borel product probability measure.
\end{corollary}

\begin{proof}
The proof is direct application of Lemma \ref{le:productmeasures}, using the observation that if $(\mu^k)_k$ converges weakly to $\mu$ then the same happens for its marginals on $X$ and $Y$.
\end{proof}

We are interested in the case when a measure has disintegrations which are product measures. For simplicity we consider products of metric spaces, but similar results can be obtained for continuous bundles. We give first a characterization of measures with product disintegrations.

\begin{lemma}[Measures with product disintegrations]\label{le:productmeasuresd}
Let $X,Y,Z$ be compact metric spaces and let $\mu$ be a Borel probability measure on $W = X \times Y \times Z$. Denote by $\mu_x$ the conditionals of $\mu$ along $\{x\} \times Y \times Z$ for $x\in X$. Then $\mu_x$ is a product measure for $(\pi_{X})_*\mu$-a.e. $x\in X$ if and only if, for every $f\in C(X,\mathbb R),\ g\in C(Y,\mathbb R)$ and $h\in C(Z,\mathbb R)$ we have
\begin{equation}\label{eq:integralproduct}
\int_W f(x)\cdot g(y)\cdot h(z)\ d\mu(x,y,z)=\int_W \left(\int_Y g \ d(\pi_{Y})_* \mu_x\right)\cdot f(x) \cdot h(z)\ d\mu(x,y,z).
\end{equation}

\end{lemma}

\begin{proof}
Consider the probability kernel induced by $\{\mu_z:z\in X\}$: that is, for $F\in C(Y \times Z, \R)$, 
\(
KF(x)=\int_{Y \times Z} F(y,z)\ d\mu_x(y,z)
\).
By the previous lemma, $\mu_x$ is a product if and only if for every $g \in C(Y,\mathbb R), h\in C(Z,\mathbb R)$ it holds
\[
	K(gh)(x)=\int_Y g \ d(\pi_{Y})_* \mu_x \cdot \int_Z h \ d(\pi_{Z})_*\mu_x; 
\]
since continuous functions are separating in $L^1(X,(\pi_{X})_*\mu)$, the previous equality holds for $(\pi_{X})_*\mu$-a.e. $x\in X$ if and only if for every $f\in C(X,\mathbb R)$ it holds
\begin{align} \label{almost_everywhere_product}
\int_X K(gh) \cdot f \ d(\pi_{X})_*\mu = \int_X f(x) \cdot \int_Y g \ d(\pi_{Y})_* \mu_x \cdot \int_Z h \ d(\pi_{Z})_*\mu_x \ d(\pi_X)_* \mu(x).
\end{align}
But $\int_Z h \ d(\pi_Z)_* \mu_x = \int_{Y \times Z} h(z) \ d\mu_x(y,z)$, so the right hand side of \eqref{almost_everywhere_product} can be written as

\begin{align*} &\int_X f(x) \cdot \left(\int_Y g \ d(\pi_{Y})_*\mu_x \int_{Y \times Z} h(z) d\mu_x(y, z)\right) \ d(\pi_{X})_*\mu(x)\\
& =\int_X  f(x) \left(\int_{Y \times Z} \int_Y g(y') d(\pi_{Y})_*\mu_x(y')\cdot h(z) \ d\mu_x(y, z)\right)  \ d(\pi_{X})_*\mu(x)\\
& =\int_W   \int_Y g \ d(\pi_{Y})_* \mu_x \cdot f(x) \cdot h(z)   \ d\mu(x,y,z).
\end{align*}
By definition of conditionals, 
\[\int_P K(fg)(z) h(z)d\pi_{P*}\mu(z)=\int_X f(x)g(y)h(z)d\mu(x,y,z),
\]
and the claim follows.
\end{proof}

We would like to pass to the limit the property of having product disintegrations. Unfortunately this is not always the case and one can easily construct examples of a sequence of measures with product disintegrations which converges in the weak* topology to a measure without product disintegrations. So we have to add some extra conditions in order to be able to pass to the limit. We have the following result.

\begin{lemma}[Limits of measures with product disintegrations]\label{le:limitproductd}
Let $X, Y, Z$ be compact metric spaces. Let $\mu^k$ be a sequence of Borel probability measures on $W =X \times Y \times Z$ such that $\mu^k$ converges to $\mu$ in the weak* topology, and for $(\pi_{X})_* \mu^k$-a.e.\@ $x \in X$, the disintegration $\mu_x^k$ is a product measure on $Y \times Z$.\\
Suppose that one of the following two conditions is verified:
\begin{enumerate}

\item
Given any $g \in C(Y, \R)$, the functions $\alpha^k(x) = \int_Y g \ d(\pi_Y)_* \mu_x^k$ and \newline $\alpha(x) =$ $\int_Y g \ d(\pi_Y)_* \mu_x$ can be extended continuously to all of $X$ and, moreover, $\alpha^k(x)$ converges uniformly to $\alpha(x)$. 

\item
The measures $(\pi_{X})_* \mu^k$ are equivalent to $(\pi_{X})_* \mu$ with the Jacobian $J^k=\frac{d(\pi_{X})_* \mu^k}{d(\pi_{X})_* \mu}$ uniformly bounded from above, and $(\pi_Y)_*\mu_x^k$ converges in the weak* topology to $(\pi_Y)_*\mu_x$ for $(\pi_{X})_*\mu$-a.e.\@ $x\in X$.
\end{enumerate}
Then the disintegrations $\mu_x$ of $\mu$ are product measures for $(\pi_{X})_* \mu$-a.e.\@ $x\in X$.
\end{lemma}

\begin{proof}
Fix $f\in C(X,\R), g \in C(Y,\R)$ and $h\in C(Z,\R)$. By Lemma \ref{le:productmeasuresd} we have
\[
	\int_W f(x)\cdot g(y)\cdot h(z)\ d\mu^k(x,y,z)=\int_W \left(\int_Y g d(\pi_{Y})_*\mu_x^k \right)\cdot f(x) \cdot h(z) \ d\mu^k(x,y,z)
\]

 The left hand side converges to \( \int_W f(x)\cdot g(y)\cdot h(z)\ d\mu(x,y,z)\) when $k$ goes to infinity hence, by the same lemma, to prove the result it suffices to show that the right hand side converges to \(\int_W \left( \int_Y g \ d(\pi_{Y})_*\mu_x\right)\cdot f(x)\cdot h(z)\ d\mu(x,y,z)\).

Consider the functions $\alpha^k, \alpha: X \rightarrow \R$,  
\begin{itemize}
	\item $\alpha^k(x)=\int_Y g \ d(\pi_{Y})_* \mu_x^k$, $\alpha(x)=\int_Y g \ d(\pi_{Y})_* \mu_x$,
\end{itemize}
which are defined for $(\pi_{X})_* \mu^k, (\pi_{X})_* \mu, $-a.e.\@ $x\in X$ respectively, and $\beta^k, \beta: W \to  \R$,
\begin{itemize}
	\item $\beta^k(x,y,z)=\alpha^k(x)\cdot f(x) \cdot h(z)$, $\beta(x,y,z)=\alpha(x)\cdot f(x) \cdot h(z)$. 
\end{itemize}
Then
\begin{align*}
\left|\int_W \beta^k\ d\mu^k-\int_W \beta\ d\mu\right|\leq &\left|\int_W \beta\ d\mu^k-\int_W \beta\ d\mu\right|\\
&+|g|_{\infty}\cdot|h|_{\infty}\int_X |\alpha^k-\alpha|\ d (\pi_X)_* \mu_k.
\end{align*}

If condition \emph{1} from the hypothesis is satisfied, then the functions $\alpha^k$ are continuous and converge uniformly to $\alpha$, hence both terms above converge to zero. On the other hand, if condition \emph{2} is satisfied, then there exists $X'\subset X$ with $(\pi_{X})_*\mu(X')=1$ such that for $x\in X', \alpha^k(x), \alpha(x)$ are well-defined and $\lim_{k\to\infty}\alpha^k(x)=\alpha(x)$. Since $\alpha^k,\alpha, J^k$ are uniformly bounded, we also get
\[
	\lim_{k\rightarrow\infty}\int_X |\alpha-\alpha_k|\ d(\pi_{X})_* \mu^k=\lim_{k\rightarrow\infty}\int_X |\alpha-\alpha_k|\cdot J^k \ d(\pi_{X})_* \mu=0.
\]
That $\left|\int_W \beta\ d\mu^k-\int_W \beta\ d\mu\right|$ converges to $0$ as $k\to\infty$ can be seen as follows. Take $(\tilde\alpha_m)_m$ a sequence of continuous functions converging in $L^1(\pi_{X\ast}\mu)$ to $\alpha$ and consider the corresponding $\tilde \beta^m=\alpha^m\cdot g\cdot h$. Then
\begin{align*}
&\left|\int_W \beta\ d\mu^k-\int_W \beta\ d\mu\right|\leq\\
&\left|\int_W \beta-\tilde\beta^m\ d\mu^k\right|+\left|\int_W \tilde\beta^m\ d\mu^k-\tilde\beta^m\ d\mu\right|+\left|\int_W \tilde\beta^m-\beta\ d\mu\right|\\
&\leq |g|_{\infty}\cdot|f|_{\infty}\cdot(\sup_k|J_k|_{\infty}+1)\int_X|\tilde\alpha^m-\alpha|\ d\pi_{X\ast}\mu+\left|\int_W \tilde\beta^m\ d\mu^k-\tilde\beta^m\ d\mu\right|.
\end{align*}
For $m$ large the first term is arbitrarily small, and then we can take $k$ large so that the second term is also small. This concludes the proof.
\end{proof}

\begin{remark}\label{re:product}
The conditions of the Lemma \ref{le:limitproductd} mean in fact that one of the two factors of the product disintegrations (corresponding to M) converges weakly to the factor of the disintegrations of the limit measure (uniformly or pointwise). However the other factor (corresponding to $N$) of the disintegrations may not converge.
\end{remark}

\subsection{Proof of Continuity.}

Before digging into the proof of \hyperlink{theoremB}{Theorem B}, let us make a tecnical comment on its proof.

Let $(f_n)$ be a sequence in $\mathcal U_1$ converging to $f\in\mathcal U_1$ in the $\mathcal C^1$ topology. For each $n$ let $\hmu_n$ be the unique $\hf_n$-invariant measure on $\hTTf[f_n]$ projecting on $\mu$. By Proposition \ref{pro:existeexpneg} we know that all these cocycles have different Lyapunov exponents at almost every point. Let $\hat\mu_n^{p-}$ be the lift of $\hat\mu_n$ to $\Pp \hat f_n$ such that the disintegrations of $\hat\mu_n^{p-}$ along the projective fibers are Dirac measures at the Lyapunov subspaces ($\delta_{E^-_n}$) for $\hat\mu_n$-a.e. point. Similarly, the corresponding objects for $f$ are  $\hTTf,\ \hat\mu,\ \Pp\hat{f}$ and $\hat\mu^{p-}$.

Denoting by $v^s(x)$ the unit vector in the Lyapunov subspace $E^-(x)$, we have
\begin{align} \label{integrated_negative_lyap_exp}
\nonumber\int_{\mathbb T^2}\chi^-(f_n)d\mu=\int_{\hTTf[f_n]}\chi^-(\hf_n)d\hmu_n&=\int_{\hTTf[f_n]}\log\|D(f_n)_{\hx}v^s\|d\hmu_n(\hat x)\\
&=\int_{\Pp \hTTf[f_n]}\log\|D(f_n)_{\hat x}v\|d\hmu_n^{p-}(\hat x,v),
\end{align}
and similarly for $f$.

To prove continuity of the negative Lyapunov exponent (which indeed implies continuity of the positive Lyapunov exponent), we would like to pass the right hand side of \eqref{integrated_negative_lyap_exp} to the limit, saying that $\hmu_n^{p-}$ converges to $\hmu^{p-}$ in the weak star topology. However, there's a technical nuisance here: the space $\hTTf$  depends on $f$, so that the measures $\hmu_n$ live on different spaces. The same is true for the measures $\hmu_n^{p-}$. For this reason, it is more convenient to work on an abstract solenoidal manifold $\Sol$ whose construction is outlined in the \hyperlink{sec:appendix}{Appendix}. Indeed, each $\hTTf[f_n]$ is homeomorphic to $\Sol$ in a canonical fashion. Thus \eqref{integrated_negative_lyap_exp} may be rewritten as 
\begin{equation} \label{integrated_negative_luap_exp_2}
\int_{\mathbb T^2}\chi^-(f_n)d\mu=\int_{\Sol}\chi^-(\fsol_n)d\bmu_n= \int_{\Pp \Sol}\|D f_n({\bx}) v\|d\bmu_n^{p-}(\bx,v),
\end{equation}
where $\bmu_n$ is the unique $\fsol_n$-invariant measure in $\Sol$ projecting on $\mu$ and $\bmu_n^{p-}$ the corresponding stable measures on $P \fsol_n$.

\begin{remark}
We could dispense with $\hTTf$ altogether and replace it with $\Sol$ throughout this paper. Indeed, $\hTTf$ can be seen simply as an embedding of $\Sol$ into $(\TT)^{\Z+}$ with the disadvantages that it depends on $f$ and that its solenoidal structure is somewhat hidden. However, we have decided to stick to $\hTTf$ in all sections but this one, due to its greater familiarity. We hope that this helps the casual reader.
\end{remark}

\begin{proof}[Proof of Theorem B]

Let $(f_n)$ be a sequence in $\mathcal U_1$ converging to $f\in\mathcal U_1$ in the $\mathcal C^1$ topology and let $\bmu_n$ and $\bmu_n^{p-}$ be the lifts to $\Sol$ and $\Pp \Sol$ as above. In virtue of \eqref{integrated_negative_luap_exp_2}, to prove Theorem B it suffices to prove that $\bmu_n^{p-}$ converges weakly* to $\bmu^{p-}$ as $n \to \infty$.
By proposition \ref{convergence_on_projetive_bundles}, this is equivalent to say that $\tbmu_n^{p-}$ converges weakly* to $\tbmu^{p-}$, where $\tbmu_n^{p-}$ and $\tbmu^{p-}$ are the lifts of $\bmu_n^{p-}$ and $\bmu^{p-}$ to $\Pp (\RR \times \Sigma)$, respectively\footnote{Recall that we are denoting $\Sigma=\{1,\cdots, d\}^{\N}$. For the definition of $\Sol$ and the different lifts of invariant measures we send the reader to the \hyperlink{sec:appendix}{Appendix}.}. 

Suppose, for the purpose of contradiction, that $\bmu_n^{p-}$ converges to some $\bmu^p$ different from $\bmu^{p-}$. We are going to show that $\bmu^p$ has product disintegrations, i.e. that it can be written as
\begin{equation} \label{structure_of_bmup} 
\bmu^p = \int_{\TT} \bmu_x \times \nu_x \ d\mu(x) 
\end{equation}
for some (measurable) family of measures $\nu_x$ on $\Pp \RR$. For that we have to analyze the convergence of $\bmu_n^{p-}$ to $\bmu^p$ in some more detail. But instead of analyzing this convergence directly, we are going to analyze the convergence of a sequence of measures on $[0,1]^2 \times \Sigma \times \Pp \RR$ which mirrors this sequence. The advantage of this is that, by doing so, the underlying space is a product space, allowing us to apply Lemma \ref{le:limitproductd}.

Let us write
\[
\begin{array}{lll}
Q = [0,1]^2, & \tbeta_n = \tbmu_n \vert Q \times \Sigma,  & \tbeta_n^{p-} = \tbmu_n^{p-} \vert Q \times \Sigma \times \Pp \RR \\
\teta = \tmu \vert Q, & \tbeta = \tbmu \vert Q \times \Sigma,  & \tbeta^{p} = \tbmu^{p} \vert Q \times \Sigma \times \Pp \RR
\end{array}
\]

Note that the boundary of $Q \times \Sigma \times \Pp \RR$ has zero $\tbmu^{p}$-measure, $(\pi_G,i)_* \tbeta_n^{p-} = \bmu_n^{p-}$, and $(\pi_G,i)_* \tbeta^p = \bmu^p$. (See the Appendix for notation.) We are now in the following situation:

\begin{enumerate}
\item
The projection of $\tbeta_n^{p-}$ to $Q \times \Sigma$ is $\tbeta_n $. The disintegrations of $\tbeta_n^{p-}$ along $ \{(\tx, \bomega) \} \times \Pp \RR$ are $\delta_{E^-_n(\pcov(\tx))}$ (see Lemma \ref{le:projective}). 
\item
The projection of $\tbeta_n$ to $Q$ is $\teta$ (the restriction of Lebesgue to the unit square). The disintegrations of $\tbeta_n$ along $\{\tx\}\times \Sigma$ are $\tbmu_{n,\tx}$ (defined as in \eqref{boldmutildextilde}).
\item
The projection of $\tbeta_n^{p-}$ to $Q$ is also $\teta$. The disintegrations of $\tbeta_n^{p-}$ along $\{\tx\} \times \Sigma \times \Pp \RR$ are $\delta_{E^-_n(\pcov(\tx))}\times\tbmu_{n,\tx}$. (This is because $E^-_n$ is constant on the fibers.)
\item
The disintegrations $\tbmu_{n,\tx}$ vary continuously with respect to  $\tx\in Q$. Furthermore  $\tbmu_{n,\tx}$ converge (uniformly) to $\tbmu_x$ (see Remark \ref{re:product} and Proposition \ref{prop:a4}).
\end{enumerate}

All these considerations show that we are within the hypothesis of Lemma \ref{le:limitproductd}, with $Q$ in place of $X$, $\Sigma$ in place of $Y$, $\Pp \RR$ in place of $Z$, $\tbeta_k^{p-}$ in place of $\mu^k$, and $\tbeta^{p}$ in place of $\mu$. (In fact both condition \emph{1} and condition \emph{2} are satisfied.) As a consequence we have that
\begin{equation} \label{tbtetap_is_intagral_of_products}
\tbeta^p=\int_{ Q}  \tbmu_{\tx}  \times \tilde{\nu}_{\tx} \ d\teta,
\end{equation}
where $\tilde{\nu}_{\tx}$, $\tx \in Q$,  are measures on $\Pp \RR$ (depending measurably on $x$). Writing $\nu_x = \tilde{\nu}_{\tx}$ when $\pcov(\tx) = x$ (which is well-defined $\teta$-a.e.) and applying $(\pi_G,i)_*$ to both sides of \eqref{tbtetap_is_intagral_of_products} gives \eqref{structure_of_bmup}. 

For the remainder of the proof we will abandon the space $\Sol$ in favor of $\hTTf$ and apply Lemma \ref{le:projective}. Thus, writing $\hmu^p = (\phi,i)_* \bmu^p$ and applying $(\phi,i)_*$ to both sides of \eqref{structure_of_bmup} we see that
\[ \hmu^p = \int_{\TT} \hat\mu_x \times \nu_x \ d\mu(x). \]

Since $\bmu_n^p$ are invariant under $PS f_n$, and $PS f_n$ converges to $PS f$, we have that $\bmu^p$ is invariant under $PS f$. Consequently, $\hmu^p$ is invariant under $P \hf$. Lemma \ref{le:projective} tells us that there exists an $\hat f$-invariant function $\hat\rho:\hTTf\rightarrow[0,1]$ such that
$$
\hat\mu^p=\int_{\hTTf}\hat\rho\delta_{E^+}+(1-\hat\rho)\delta_{E^-}\ d\hat\mu.
$$

Since the ergodic components of $\hat f$ are exactly the pre-images under $\pi_{ext}$ of the ergodic components of $f$, we have in fact that $\hat\rho$ is constant on the fibers $\pi_{ext}^{-1}(x)$, in other words there exists a measurable $f$-invariant function $\rho:\mathbb T^2\rightarrow[0,1]$ such that $\hat\rho=\rho\circ\pi_{ext}$. Remembering that $E^-$ is also constant on the fibers, we have that
\begin{eqnarray*}
\int_{\hTTf}\rho\circ\pi_{ext}\delta_{E^+}d\hat\mu&=&\hat\mu^p-\int_{\hTTf}(1-\rho\circ\pi_{ext})\delta_{E^-}d\hat\mu\\
&=&\int_{\mathbb T^2}\nu_x\times\hmu_xd\mu-\int_{\mathbb T^2}(1-\rho)\delta_{E^-}\times\hmu_xd\mu\\
&=&\int_{\mathbb T^2}\left(\nu_x-(1-\rho)\delta_{E^-}\right)\times\hmu_xd\mu\\
&=&\int_{\hTTf}\nu_{\pi_{ext}(\hat x)}-(1-\rho\circ\pi_{ext})\delta_{E^-}d\hat\mu.
\end{eqnarray*}

Then for $\hat\mu$ almost every $\hat x$ we have that 
\[\rho(\pi_{ext}(\hat x))\delta_{E^+(\hat x)}=\nu_{\pi_{ext}(\hat x)}-(1-\rho(\pi_{ext}(\hat x)))\delta_{E^-(\hat x)}.
\]
If $\hat\mu^p$ is different from $\hat\mu^{p-}$ then $\rho$ is nonzero on a set of positive measure. Then there exist a $\mu$-positive measure set of points $x\in\mathbb T^2$ such that $\rho(x)>0$ and $\rho(x)\delta_{E^+(\hat x)}=\nu_{x}-(1-\rho(x))\delta_{E^-(\hat x)}$ holds for $\hmu_x$-a.e $\hat x\in\pi_{ext}^{-1}(x)$. Since the right-hand side depends only on $x$, we have that for $\hmu_x$-a.e. $\hat x$, the unit vector $v^+(\hat x)$ inside $E^+(\hat x)$ is constant, equal to say $v^+(x)$. Since almost every $x\in\TT$ has the property that $\hmu_x$-a.e. point is Lyapunov regular, then there exists a set $A\subset\TT$ of positive $\mu$-measure such that for every $x\in A$ we have
\begin{eqnarray*}
\lim_{n\rightarrow\infty}\frac {I(x,v^+;f^n)}n&=&\lim_{n\rightarrow\infty}\frac 1n\int_{\pi_{ext}^{-1}(x)}\log\|(D_{\hat x}\hat f)^{-n}(v^+)\|d\hmu_x\\
&=&\int_{\pi_{ext}^{-1}(x)}\lim_{n\rightarrow\infty}\frac 1n\log\|(D_{\hat x}\hat f)^{-n}(v^+)\|d\hmu_x\\
&=&\int_{\pi_{ext}^{-1}(x)}-\chi^+d\hmu_x.
\end{eqnarray*}

This implies that $\mu$-a.e. $x\in A$ we have $$\chi^+(x)=-\lim_{n\rightarrow\infty}\frac {I(x,v^+;f^n)}n\leq -C_{\chi}(f).$$

But on the other hand we know from \ref{pro:existeexpneg} that $\mu$-a.e. $x\in\TT$ we have
$$
\chi^-(x)\leq-C_{\chi}(f),
$$
so for $\mu$-a.e. $x\in A$ we have
$$
-2C_{\chi}(f)\geq\chi^+(x)+\chi^-(x)=\lim_{n\rightarrow\infty}\log(\det Df^n(x))\geq C_{\det}(f),
$$
which is a contradiction with the definition of $\mathcal U_1$. This finishes the proof.
\end{proof}

%%%%%%%%%%%%%%%%%%%%%%%%%%%%%%%%%%%%%%%%%%%%%%%%%%%%%%%%%%%%%%%%%%%%%%%%%%%%%%%%%%%%%%%%%%%%%%%%%%%%%%%%%%%%%%%%%%%%%%
\section{Pesin theory for endomorphisms}
\label{sec:Pesintheory}
%%%%%%%%%%%%%%%%%%%%%%%%%%%%%%%%%%%%%%%%%%%%%%%%%%%%%%%%%%%%%%%%%%%%%%%%%%%%%%%%%%%%%%%%%%%%%%%%%%%%%%%%%%%%%%%%%%%%%%
%%%%%%%%%%%%%%%%%%%%%%%%%%%%%%%%%%%%%%%%%%%%%%%%%%%%%%%%%%%%%%%%%%%%%%%%%%%%%%%%%%%%%%%%%%%%%%%%%%%%%%%%%%%%%%%%%%%%%%

We collect here some facts about non-uniformly hyperbolic endomorphisms that will be used in the forthcoming section to prove (stable) ergodicity. Throughout this part $f:M\to M$ denotes an endomorphism of the compact surface $M$; the distance on $M$ (induced by a Riemannian metric) between $x,y\in M$ is denoted by $d(x,y)$. It is assumed that $f$ is non-uniformly hyperbolic with respect to an smooth area $\mu$, and we maintain the notation $\hTTf, \hf,\hmu$, etc. as before (see part \ref{sub:naturalext}). We point out that in our notation $x_n$ (and not $x_{-n}$) is a preimage of $x_0$ for $f^n$. Recall also the definition of the set or regular points $\reg$  (cf.\@ $\hat \reg$) given in Lemma \ref{lem:exponentes}. We will use the following definitions of \cite{Liu2008} and \cite{pesinendo}.

\begin{definition}
 Let $\hx\in \hat\reg, x=\pext(\hx)$. 
\begin{enumerate}
    \item A $\mathcal C^{1,1}$ embedded interval $\Wuloc{\hx} \subset M$ is a local unstable manifold at $\hx$ if there exist constants $\lambda>0, 0<\varepsilon< \frac{\lambda}{200}, 0<C_1\leq 1<C_2$ so that $y_0\in \Wuloc{\hx}$ if and only if there exists a unique $\hy\in\hTTf, y_0=\pext(\hx)$ satisfying, for every $n\geq 0$
    \begin{align}\label{eq:localWu}
      d(x_n,y_n)\leq C_1 e^{-n\varepsilon}\\
      d(x_n,y_n)\leq C_2 e^{-n\lambda}.
    \end{align}
    The lift of $\Wuloc{\hx}$ to $\hTTf$ using $\pext$ is denoted by $\hWuloc{\hx}$.
    \item The unstable manifold of $f$ at $\hx$ is
    \[
      W^u(\hx)=\{y_0=\pext(\hy):\limsup_{n\to\infty}\frac{1}{n}\log d(x_n,y_n)<0\}.
    \]
    The lift of $W^u(\hx)$ to $\hTTf$ using $\pext$ is denoted by $W^u(\hx)$.
 
    \item A $\mathcal C^{1,1}$ embedded interval $\Wsloc{x} \subset M$ is a local stable manifold at $x$ if there exist constants $\lambda>0, 0<\varepsilon< \frac{\lambda}{200}, 0<C_1\leq 1<C_2$ so that $y\in \Wsloc{x}$ if and only if  for every $n\geq 0$
    \begin{align}\label{eq:localWs}
      d(f^n(x),f^n(y))\leq C_1 e^{-n\varepsilon}\\
      d(f^n(x),f^n(y))\leq C_2 e^{-n\lambda}.
    \end{align}
    The lift of $\Wsloc{x}$ to $\hTTf$ using $\pext$ is denoted by $\hWsloc{\hx}$.

    \item The stable manifold of $f$ at $x$ is
    \[
    W^s(x)=\pext\Big(\bigcup_{n=0}^{\infty} \hf^{-n}\hWsloc{\hf^n \hat x}\Big).
    \]
\end{enumerate}
\end{definition}

\begin{remark}
 Observe that 
  \[
  W^s(x) \subset \{y: \limsup \frac{1}{n}d(f^nx,f^ny)<0\}=\bigcup_{k=0}^{\infty} f^{-k}(f^kW^s(x)).
 \]
 \end{remark}

We record the following basic properties of stable and unstable manifolds.

\begin{theorem}[Proposition $2.3$ in \cite{pesinendo}]
  There exists an increasing countable family $\{\hat\Lambda_k\}_{k\geq 0}$ of subsets of $\hat\reg$ satisfying that 
  $\hmu(\bigcup_k \hat\Lambda_k)=1$, and such that:
 \begin{enumerate}
    \item For every $k$ there exists a continuous family of local unstable manifolds $\{\Wuloc{\hx}:\hx\in \hat \Lambda_k\}$  so that for every $\hx\in\hat\Lambda_k$ it holds:
    \begin{enumerate}
      \item $T_{x_0}\Wuloc{\hx}=E^{+}(\hx)$; in particular $E^{+}(\hx)$ depends continuously on $\hx\in\hat\Lambda_k$.
      \item There exists a sequence of $\mathcal C^{1,1}$ intervals $\{W^u(\hx,-n)\}_{n\geq 0}$ in $M$ with
      \begin{itemize}
        \item $W^u(\hx,0)=\Wuloc{\hx}$,
        \item $fW^u(\hx,-n)\supset W^u(\hx,-n+1), \forall n\geq 1$, 
        \item $W^u(\hx)=\bigcup_{n\geq 0} f^nW^u(\hx,-n)$.
      \end{itemize}
    \end{enumerate}
    \item If $\Lambda_k=\pext(\hat \Lambda_k)$, then:
    \begin{enumerate}
      \item For every $k$ there exists a continuous family of local stable manifolds $\{\Wsloc{x}:x\in \Lambda_k\}$ so that for every $x\in \Lambda_k$ it holds:
      \begin{itemize}
        \item $T_{x}\Wsloc{x}=E^{-}(x)$; in particular $E^{-}(x)$ depends continuously on $x\in\Lambda_k$.
        \item  $f\Wsloc{x} \subset \Wsloc{fx}$
      \end{itemize}
    \end{enumerate}
  \end{enumerate}
\end{theorem}

\begin{proof}
See \cite{pesinendo} (Proposition $2.3$) and \cite{Liu2008}.
\end{proof}

\begin{definition}
The sets $\Lambda_k, \hat \Lambda_k$ above will be referred as \emph{Pesin blocks}.
\end{definition}

It follows in particular that $\Wuloc{\hx}$ has diameter uniformly bounded from below, for points in the same $\hat\Lambda_k$, while the local stable manifolds have sizes uniformly bounded from below, for points in the same $\Lambda_k$. Also, $f\Wsloc{x}\subset\Wsloc{fx}$; it follows that $W^s(x)$ is an immersed submanifold of $M$, and moreover these  manifolds $W^s$ (or $\hat{W}^s$) form an invariant lamination of $M$ (respectively, of $\hTTf$). The global unstable manifolds $W^u$ do not form a lamination of $M$ (they may intersect), however $\hat{W}^u$ do form an invariant lamination of $\hTTf$.

\smallskip

We will also need the following two results.

\begin{proposition}(see \cite{Liu2008})\label{pro:sabscont}
The stable lamination $W^s$ is absolutely continuous in the following sense. Given any Pesin block $\Lambda_k$, the holonomy of local stable manifolds of points in $\Lambda_k$ between any two transversals is absolutely continuous (with respect to the Lebesgue measure of the two transversals). As a consequence, given any partition of $\TT$ subordinated to the stable lamination, the disintegration of the Lebesgue measure on $\TT$ along the elements of the partition are absolutely continuous with respect to the Lebesgue measures on stable manifolds.

\end{proposition}

\begin{proposition}(see \cite{Liu2011})\label{pro:abscontWu}
The unstable lamination $\hat W^u$ is absolutely continuous in the following sense. Given any partition of $\hTTf$ subordinated to the Pesin unstable lamination $\hat W^u$, the disintegrations $\hmu$ along the elements of the partition are absolutely continuous with respect to the Lebesgue measure on the unstable manifolds.
\end{proposition}

Here the Lebesgue measure on $\hat W^u(\hx)$ is given by the Lebesgue measure on $W^u(\hx)$, since $\pext$ is a local homeomorphism between the two unstable manifolds.

%%%%%%%%%%%%%%%%%%%%%%%%%%%%%%%%%%%%%%%%%%%%%%%%%%%%%%%%%%%%%%%%%%%%%%%%%%%%%%%%%%%%%%%%%%%%%%%%%%%%%%%%%%%%%%%%%%%%%%%
\section{Stable ergodicity}
%%%%%%%%%%%%%%%%%%%%%%%%%%%%%%%%%%%%%%%%%%%%%%%%%%%%%%%%%%%%%%%%%%%%%%%%%%%%%%%%%%%%%%%%%%%%%%%%%%%%%%%%%%%%%%%%%%%%%%%
%%%%%%%%%%%%%%%%%%%%%%%%%%%%%%%%%%%%%%%%%%%%%%%%%%%%%%%%%%%%%%%%%%%%%%%%%%%%%%%%%%%%%%%%%%%%%%%%%%%%%%%%%%%%%%%%%%%%%%%

\subsection{Large stable manifolds}\label{sub:Large}

The goal of this subsection is to prove that almost every (Lyapunov regular) point for the examples constructed in Section \ref{sec:shears} has a large stable manifold, in terms of diameter. We keep the notations from that section: $f_t=h_t\circ E, \delta<\frac 1{10}, \mathcal G, \alpha>1, \Delta_{\alpha}^{v,h}, 0<a<b, e_v$ and $e_h$.

If $\gamma:I\subset\mathbb R\rightarrow\TT$ is a $\mathcal C^1$ curve, then $l(\gamma)$ is its (euclidean) length. Since in Section \ref{sec:shears} we did the estimates in the maximum norm, we will also consider the length of the curve in the maximum norm:
$$
l_m(\gamma):=\int_I\max\{|\gamma_1'(t)|,|\gamma_2'(t)|\}dt.
$$
We clearly have $l_m(\gamma)\leq l(\gamma)\leq \sqrt 2l_m(\gamma)$.

Let $l=\frac {\alpha}{5e_v}$.

\begin{definition}
A \textit{v}-segment is a $\mathcal C^1$ curve on $\TT$ which is tangent to the vertical cone $\Dev$, and whose length in the maximum norm is equal to $l$.
\end{definition}

\begin{remark}
Since $\alpha>1$, if $\gamma$ is a \textit{v}-segment then in fact the length of the projection of $\gamma$ on the vertical axis is $l$.
\end{remark}

\begin{lemma}
Assume that $t>\max\{\frac{2\alpha}a,\frac{4\alpha^2+\alpha e_v}{e_va}\}$. There exists $\mathcal U_2$ a $\mathcal C^1$ open neighborhood of $f_t$ such that any $f\in \mathcal U_2$ satisfies the following property: if $\gamma$ is a curve in $\TT$ such that $f(\gamma)$ is a \textit{v}-segment, then $\gamma$ contains a \textit{v}-segment.
\end{lemma}

In other words, any pre-image of a \textit{v}-segment contains a \textit{v}-segment.

\begin{proof}

From Section \ref{sec:shears} we have:
\begin{enumerate}
\item
Action of $E^{-1}$: $E^{-1}\Dev\subset \Deh$ and $\conorm{E^{-1}|_{\Dev}}=e_v$ ($E^{-1}$ takes $\Dev$ strictly inside $\Deh$ expanding by more than $e_v$);
\item
Action of $Dh_t^{-1}$: if $t>\frac{2\alpha}a$ then for any $x\in\overline{\mathcal G}$ we have $D_xh_t^{-1}\Deh\subset\overline{D_xh_t^{-1}\Deh}\subset \Dev$ and $\conorm{D_xh_t^{-1}|_{\Deh}}>\frac{at-\alpha}{\alpha}$ ($D_xh_t^{-1}$ takes $\Deh$ strictly inside $\Dev$ expanding by strictly more than $\frac{at-\alpha}{\alpha}$);
\end{enumerate}

The second property is clearly $\mathcal C^1$ open, so we choose
\begin{align*}
\mathcal U_2:=&\{f=h\circ E:\ h \hbox{ is a }\mathcal C^1\hbox{ diffeomorphism of }\TT\hbox{ satisfying the above}\\ &\hbox{ property 2}\}.
\end{align*}

Let $\gamma$ be a curve in $\mathbb R^2$ such that $f(\gamma)$ is a \textit{v}-segment. Then $h(\gamma)$ is a pre-image of $f(\gamma)$ under some inverse branch of $E$, so it is a curve inside the horizontal cone $\Delta_{\alpha}^h$ such that $l_m(h(\gamma))\geq e_vl=\frac{\alpha}5$. Since $h(\gamma)$ is inside $\Deh$ we obtain that $l(\pi_1h(\gamma))\geq\frac 15$ (the length of the projection of $h(\gamma)$ on the horizontal direction is at least $\frac 15$).

Recall that the critical intervals $I_1$ and $I_3$ have sizes $2\delta<\frac 1{2\tau_2}\leq \frac 1{10}$, this implies that there exists a piece $\gamma_1$ inside $h(\gamma)$ which is inside $\mathcal G$ and has the length of the projection on the horizontal direction strictly greater than $\frac 1{20}$. Then our hypothesis implies that $h^{-1}(\gamma_1)$ will be inside the vertical cone $\Dev$ and $l_m(h^{-1}(\gamma_1))\geq\frac {at-\alpha}{20\alpha}$. Since $t>\frac{4\alpha^2+\alpha e_v}{e_va}$ we obtain that $l_m(h^{-1}(\gamma_1))\geq\frac {\alpha}{5e_v}=l$, so $h^{-1}(\gamma_1)\subset\gamma$ contains a \textit{v}-segment.
\end{proof}

\begin{lemma}
Assume that $t>\max\{\frac{2\alpha}a,\frac{4\alpha^2+\alpha e_v}{e_va}\}$. There exists $\mathcal U_3\subset\mathcal U_2$ a $\mathcal C^1$ open neighborhood of $f_t$ such that any $f\in \mathcal U_3$ satisfies the following property: For any $x\in\TT$ and any $\mathcal C^1$ curve $\gamma$ passing through $x$, there exists $n\in\N$, $y\in \TT$ and a \textit{v}-segment $\gamma'$ passing through $y$ such that $f^n(y)=x$ and $f^n(\gamma')=\gamma$ (some preimage of $\gamma$ contains a \textit{v}-segment).
\end{lemma}

\begin{proof}

We recall that  Lemma \ref{lem:preimagenesdd} tells us that $f_t$ satisfies the following two properties:
\begin{enumerate}
\item
For any $(x,v)\in T\TT$, there exists a preimage $(y,w)\in \mathcal G\times \Dev$
\item
For any $x\in\TT$, there exists a preimage $y\in\mathcal G$ such that $d(y,\mathcal C)>\frac 1{10}$.
\end{enumerate}

These are clearly $\mathcal C^1$ open properties, so we choose $\mathcal U_3\subset\mathcal U_2$ such that every $f\in\mathcal U_3$ satisfies them.

Let $x\in\TT$ and $\gamma$ a $\mathcal C^1$ curve passing through $x$, and let $v$ be the vector tangent to $\gamma$ in $x$. There exists a pre-image $x_1$ of $x$ such that $v_1=(D_{x_1}f_t)^{-1}v\in\Dev$ and $x_1\in\mathcal G$. Furthermore we can construct inductively a sequence of pre-images $(x_n)_{n\in\mathbb N}$ of $x$ such that
\begin{enumerate}
\item
$f(x_{n+1})=x_n$, $f^n(x_n)=x$;
\item
$x_n\in\mathcal G$, $d(x_n,\mathcal C)>\frac 1{10}$, for all $n\geq 2$
\end{enumerate}
Let $\gamma'_1$ be the preimage of $\gamma$ passing through $x_1$. Since the vector tangent to $\gamma'_1$ in $x_1\in\mathcal G$ is $v_1\in\Dev$, and $\gamma'_1$ is $\mathcal C^1$, we can find a curve $\gamma_1$ inside $\gamma'_1$ which contains $x_1$, is tangent to the vertical cone $\Dev$ and is inside $\mathcal G$. Taking the corresponding pre-images of $\gamma_1$ we obtain a sequence of curves $\gamma_n$.

If all $\gamma_n$ stay inside $\mathcal G$, then they will stay tangent to the vertical cone and will grow exponentially, so we eventually obtain a \textit{v}-segment.

If not, there exists $n\in\mathbb N$ such that $\gamma_{n+1}$ is not in $\mathcal G$ and $\gamma_n=f(\gamma_{n+1})$ is in $\mathcal G$. But this implies that the length of the horizontal projection of the part of $\gamma_{n+1}$ inside $\mathcal G$ is greater than $\frac 1{10}$, and the considerations from the previous lemma imply that $\gamma_{n+1}$ must contain a \textit{v}-segment. This finishes the proof.

\end{proof}

Let us consider now $f\in\mathcal U\cap\mathcal U_3$, so $f$ is $\mathcal C^{1+\epsilon}$, preserves the area and is nonuniformly hyperbolic. Let $A=\{x\in\TT:\ W^s(x) \hbox{ contains a \textit{v}-segment}\}$. The two previous lemmas show that:
\begin{itemize}
\item
$f^{-1}(A)\subset A$;
\item
$\bigcup_{n=0}^{\infty}f^n(A)$ has full Lebesgue measure.
\end{itemize}
The next lemma says that a set with these properties has full measure.

\begin{lemma}
Assume that $f$ preserves the probability measure $\mu$, and the set $A$ satisfies $f^{-1}(A)\subset A$ and $\mu(\bigcup_{n=0}^{\infty}f^n(A))=1$. Then $A$ has full $\mu$-measure.
\end{lemma}

\begin{proof}
Observe that the second condition implies that $A$ has positive measure.

Let $B\subset A$ be the set of recurrent points, $B=\{x\in A:\ \forall n\in\N,\ \exists m\in\N, m>n, f^m(x)\in A\}$. Then $\mu(B)=\mu(A)$, and therefore $\bigcup_{n=0}^{\infty}f^n(B)$ has full measure.

We claim that $\bigcup_{n=0}^{\infty}f^n(B)\subset A$. Let $x\in \bigcup_{n=0}^{\infty}f^n(B)$: then there exists $n\in\N$ such that $f^n(x)\in A$, and therefore the first property gives us that $x\in A$. This finishes the proof.
\end{proof}

Putting the three lemmas above together we obtain the following proposition.

\begin{proposition}\label{cor:largestable}
For any $t>\max\{\frac{2\alpha}a,\frac{4\alpha^2+\alpha e_v}{e_va}\}$ there exists a $\mathcal C^1$ open set of area preserving nonuniformly hyperbolic endomorphisms $\mathcal U_3\cap \mathcal U$ containing $f_t$ such that any $f\in \mathcal U_3\cap\mathcal U$ of class $C^{1+\epsilon}$ satisfies the following property: Lebesgue almost every point $x\in\TT$ has large stable manifold, in the sense that $W^s(x)$ contains a \textit{v}-segment.
\end{proposition}

\color{black}

\subsection{Local ergodicity}

In this subsection we prove a general result from which both \hyperlink{theoremC}{Theorem C} and \hyperlink{theoremD}{Theorem D} can be deduced. 

\begin{proposition} \label{local_ergodicity}
	Let $f:\TT \to \TT$ be a conservative non-uniformly hyperbolic endomorphism. Suppose there exists $\delta>0$ such that the global stable manifold of almost every point has a diameter greater than $\delta$. Then $f$ is locally ergodic.
\end{proposition}

Here, a conservative map $f$ is called \emph{locally ergodic} if its ergodic components are open modulo sets of measure zero. That is, there exists a finite collection of open sets $U_i$ such that 
\begin{enumerate}
	\item $\mu_i = \frac{\mu \vert U_i}{\mu(U_i)}$ is ergodic for every $i$, and
	\item $\mu(\bigcup_{i} U_i) = 1$.
\end{enumerate}

Recall (see section \ref{sub:naturalext}) that $\hTTf$ is the natural extension of $f$, with the lift $\hf$ of $f$ and the lift $\hmu$ of $\mu$. Let $\hat\reg$ be the set of Lyapunov regular points of $\hTTf$, which has full $\hmu$ measure. The set of regular points
 $\reg=\pext(\hat\reg)$ is a full area set in $\mathbb T^2$, and $\reg=\bigcup_{k=0}^{\infty}\Lambda_k$ is the corresponding filtration into Pesin blocks.

Let $\hat\reg_0\subset\hat\reg$ be the points $\hx\in \reg$ with the property that Birkhoff averages of the Dirac measure $\delta_{\hx}$ converge to the same limit for both $\hf$ and $\hf^{-1}$. Then $\hat\reg_0$ has full $\hmu$ measure, and consists in fact of the points which are situated inside the basins of attraction (both forward and backward) of the measures which are in the ergodic decomposition of $\hmu$. Let $\hat\reg_1\subset\hat\reg$ be the points with the property that for any $\hx\in\hat\reg_1$, Lebesgue almost every point in $\hat W^u(\hx)$ is in $\hat\reg_0$. The absolute continuity of the unstable foliation \ref{pro:abscontWu} implies that $\hat\reg_1$ has full $\hmu$ measure. Both $\hat\reg_0$ and $\hat\reg_1$ are invariant under $\hat f$. Furthermore $\reg_0=\pext(\hat\reg_0)$ and $\reg_1=\pext(\hat\reg_1)$ are forward invariant under $f$ and have full Lebesgue measure on $\TT$.

\begin{definition}
An $su$-rectangle is a piecewise smooth simple closed curve in $\mathbb T^2$  consisting of two pieces of local stable manifolds and two pieces of local unstable manifolds.

An $su$-rectangle is $\mu_0$-regular if the two pieces of unstable manifolds are contained in $\hat \reg_1$ and Lebesgue almost every point in the unstable pieces are in the basin of the measure $\mu_0$.

The ergodic Pesin block $\Lambda_{k, \mu_0}$ is the intersection of $\Lambda_{k}$ with the basin of $\mu_0$.
\end{definition}

The next lemma says that if $f$ is conservative non-uniformly hyperbolic then the ergodic components of positive Lebesgue measure cover all the manifold. It is a well-known result and it goes back to Pesin (at least for diffeomorphisms); we include a proof for the convenience of the reader. Compare with \cite{Pugh1989}.

\begin{lemma}\label{le:blocks}
Suppose that the Lebesgue measure is hyperbolic, then Lebesgue almost every point in $\TT$ belongs to an ergodic Pesin block of positive Lebesgue measure.
\end{lemma}

\begin{proof}
We have that
\[
\mu\left(\bigcup_{\mu(\Lambda_k)>0}\Lambda_k\right)=1.
\]

Assume that $\Lambda_k$ has positive Lebesgue measure, and let $\Lambda'_k$ be the points of $\Lambda_k$ which are Lebesgue density points of $\Lambda_k$; then $\mu(\Lambda'_k)=\mu(\Lambda_k)$. We will show in fact that all the points in $\bigcup_{\mu(\Lambda_k)>0}\left(\Lambda_k'\cap\Lambda_k\cap\reg_1\right)$ satisfy the conclusion of the Lemma.

Let $x\in\Lambda_k'\cap\Lambda_k\cap\reg_1$. Then there exists $\hx\in\hat\reg_1$ such that $\pext(\hx)=x$. We have that $W^u_{loc}(\hx)$ is transversal to $W^s_{loc}(x)$, so for some small $r>0$ and for every $y\in B(x,r)\cap\Lambda_k$ we have that $W^u_{loc}(\hx)$ is transverse to $W^s_{loc}(y)$.

Since $\mu(B(x,r)\cap\Lambda_k)>0$ and the local stable lamination is absolutely continuous, we have that the Lebesgue measure (along $W^u_{loc}(\hx)$) of $\bigcup_{y\in B(x,r)\cap\Lambda_k}\left(W^u_{loc}(\hx)\cap W^s_{loc}(y)\right)$ is positive.

Since $\hx\in\hat\reg_1$ we have that Lebesgue almost every point in $W^u_{loc}(\hx)$ is in the basin of the same ergodic component $\mu_0$ of $\mu$. Since the basin of $\mu_0$ is saturated by stable manifolds, then in must contain Lebesgue almost every point in $B(x,r)\cap\Lambda_k$, so we get that $\Lambda_{k,\mu_0}$ has positive Lebesgue measure and $x$ is a density point of $\Lambda_{k,\mu_0}$.

Now, in order to finish the proof, we just observe that $$\bigcup_{\mu(\Lambda_k)>0}\left(\Lambda_k'\cap\Lambda_k\cap\reg_1\right)$$ has full Lebesgue measure on $\TT$.
\end{proof}

\begin{remark}
We have that the ergodic components of $\mu$ with positive Lebesgue measure cover almost all the torus. Furthermore, if $\mu_0$ is an ergodic component of $\mu$ with positive Lebesgue measure, then the basin of $\mu_0$ contains (and is in fact equal to, modulo zero measure) $\bigcup_{k>0}\bigcup_{x\in\Lambda_{k,\mu_0}}W^s(x)$.
\end{remark} 

\begin{lemma}\label{le:rectangles}
Suppose that the Lebesgue measure is hyperbolic, then Lebesgue almost every point in $\TT$ is in the interior of an arbitrarily small regular $su$-rectangle.
\end{lemma}

\begin{proof}
We consider the regular points with the property that almost all the points on the local stable manifold are in $\reg_1$:
\[
\reg_2=\{x\in\reg:\ \Leb_{W^s_{loc}(x)}(W^s_{loc}(x)\setminus\reg_1)=0\}
\]
where $\Leb_{W^s_{loc}(x)}$ denotes the induced Lebesgue measure on the submanifold $\Wsloc{x}$. Since the stable manifolds form an absolutely continuous lamination, the conditional measures of $\mu$ with respect to any measurable partition subordinated to the stable lamination are equivalent to the corresponding induced Lebesgue measure (see Proposition \ref{pro:sabscont}), hence we have that $\reg_2$ has full Lebesgue measure. Let $x\in\Lambda_k'\cap\Lambda_k\cap\reg_2$, for some $k>0$ with $\mu(\Lambda_k)>0.$ 

Since $x\in\reg_2$, there exist $y,z\in W^s_{loc}(x)\cap\reg_1$ such that $y$ and $z$ are arbitrarily close to $x$ and $x$ is between $y$ and $z$. There exist $\hat y,\hat z\in\hat\reg_1$ such that $y=\pext(\hat y), z=\pext(\hat z)$. Furthermore $\Wuloc{\hat y}$ and $\Wuloc{\hat z}$ are transverse to $\Wsloc{x}$.

This implies that for all sufficiently small $r>0$ we have that for all $x'\in\Lambda_k\cap B(x,r)$, $\Wsloc{x'}$ is transverse to $\Wuloc{\hat y}$ and $\Wuloc{\hat z}$. We know that $x$ is a density point of $\Lambda_k$, and $\Lambda_k\cap B(x,r)$ has positive measure, so we have a positive measure set of stable manifolds inside the Pesin block $\Lambda_k$ cutting transversely both $\Wuloc{\hat y}$ and $\Wuloc{\hat z}$. The stable lamination is absolutely continuous, and $\hat y,\hat z\in\reg_1$, so we conclude that both $\hat y$ and $\hat z$ are in the basin of the same ergodic component $\mu_0$ of $\mu$.

Since $x$ is a density point of $\Lambda_k$, we can indeed construct an arbitrarily small $su$-rectangle with $x$ in its interior and the unstable pieces in $\Wuloc{\hat y}$ and $\Wuloc{\hat z}$. In consequence, $x$ is in the interior of an arbitrarily small regular $su$-rectangle.

In order to conclude the proof we remark that $\bigcup_{k>0}(\Lambda_k'\cap\Lambda_k\cap\reg_2)$ has full Lebesgue measure on $\TT$.
\end{proof} \color{black}

\begin{lemma}\label{le:interior}
Suppose that the Lebesgue measure is hyperbolic and almost every stable manifold has diameter larger than $\delta>0$. Given a $\mu_0$-regular $su$-rectangle, with diameter smaller than $\delta$, then Lebesgue almost every point in the interior of it is also inside the basin of $\mu_0$.
\end{lemma}

\begin{proof}
Let $R=W^s_1\cup W^u_1\cup W^s_2\cup W^u_2$ be a regular $su$-rectangle of diameter smaller than $\delta$, formed by the stable segments $W^s_1$ and $W^s_2$ and the unstable segments $W^u_1$ and $W^u_2$. Let $V$ be the interior of $R$.

From Lemma \ref{le:blocks} we know that we can cover almost all $V$ with ergodic Pesin blocks of positive measure, so what we have to prove is that all these blocks correspond to the same ergodic component $\mu_0$ of $R$.

Suppose that this is not true: then there exists an ergodic component $\mu_1$ of $\mu$ different than $\mu_0$ such that $B=\Lambda_{k,\mu_1}\cap V$ has positive Lebesgue measure. Let $B'$ be the set of Lebesgue density points of $B$ and $\mathrm{Rec}(B\cap B')$ the recurrent points of $B\cap B'$, i.e. the points coming back to $B\cap B'$ infinitely many times. Then $\mathrm{Rec}(B\cap B')$ has full measure in $B$, so we can find a point $x\in \mathrm{Rec}(B\cap B')$ with the property that the diameter of its stable manifold is larger than $\delta$. This implies that $W^s(x)$ intersects one of the unstable pieces of the $su$-rectangle, say $W^u_1$. Let $y\in W^s(x)\cap W^u_1$ be a point where the intersection is topologically transversal.

The point $x$ is recurrent to $B\cap B'$, and the piece of stable manifold between $x$ and $y$ is contracting under iterates of $f$. Then there exists $n\geq 0$ such that $f^n(x)\in B\cap B'$ and $f^n(y)\in W^s_{loc}(f^n(x))$. Furthermore the intersection between $f^n(W^u_1)$ and $\Wsloc{f^n(x})$ is topologically transversal.

In conclusion $f^n(W^u_1)$ is a $\mathcal C^1$ curve which crosses transversely a local stable manifold of a density point in $B=\Lambda_{k,\mu_1}\cap V$. We are now in the condition to apply the same argument of Proposition $5.1$ in \cite{Hertz2011} for the lamination of local stable manifolds of points in $B$ and the curve $f^n(W^u_1)$ and from there we obtain that there exists a Lebesgue positive set of points in $f^n(W^u_1)$ which intersect this local stable lamination.

Let us explain briefly this argument. The lamination of stable manifolds of points in $B$ can be extended locally to a $\mathcal C^1$ foliation box. We can choose a smooth transversal $T$ to the foliation, and sliding along the holonomy we can define a map from a piece of $f^n(W^u_1)$ to $T$. The map is $\mathcal C^1$ and Sard's Theorem implies that the critical values have zero Lebesgue measure. Critical values correspond to tangencies between $f^n(W^u_1)$ and the foliation, so there is a full measure set of leaves which are transverse to $f^n(W^u_1)$ in a small neighborhood of the density point. Because we have a density point of the lamination and the foliation is absolutely continuous, then we obtain a positive set of leaves of the lamination (local stable manifolds) which will intersect $f^n(W^u_1)$ in a set of positive Lebesgue measure.
 
But this means that a Lebesgue positive set of points in $f^n(W^u_1)$ are in the basin of $\mu_1$, and this is a contradiction because we know from hypothesis that Lebesgue almost all the points in $W^u_1$ are in the basin of $\mu_0$. This finishes the proof.
\end{proof}

\begin{proof}[Proof of Proposition \ref{local_ergodicity}]
	Let $\mathbf{R}$ be the collection of all regular $su$-rectangles. Let $\sim$ be the equivalence relation on $\mathbf{R}$ defined by $R_1 \sim R_2$ iff $R_1 \cap R_2$ has non-empty interior. Denote by $\{\mathbf{R}_i\}_{i \in I}$ the equivalence classes of $\mathbf{R}$ and write
	\[ U_i = \bigcup_{R \in \mathbf{R}_i} R.\]
	Then $\{U_i\}_{i \in I}$ is a family of non-empty pairwise disjoint open set, so there are at most countably many of them. By Lemma \ref{le:rectangles}, $\mu(\bigcup_{i \in I} U_i) = 1$, and by Lemma \ref{le:interior}, the normalized restriction of $\mu$ to each $U_i$ is ergodic.
\end{proof}

\subsection{Bernoulli property} % (fold)
\label{sub:bernoulli_property}

Now we will explain how to deduce a stronger property than ergodicity for the maps considered in \hyperlink{theoremD}{Theorem D}, namely, the Bernoulli property. 

\begin{definition}
Let $f:M\to M$ be an endomorphism preserving a measure $\mu$, and consider the corresponding lifts $\hat f:\hTTf\to\hTTf$ and $\hat \mu$. We say that $(f,\mu)$ is Bernoulli if $(\hat f,\hat \mu)$ is metrically isomorphic to a Bernoulli process.
\end{definition}

For the rest of this part we fix $f$ as in the hypothesis of \hyperlink{theoremD}{Theorem D}: in particular, it is $C^2$, NUH and ergodic for the area. To take advantage of previous technology for diffeomorphisms, we will use the existence of a smooth model for $\hTTf$; that is, there exists a smooth manifold $L=M\times D^n$, where $D^n$ is a $n$-dimensional disk for some $n>0$, and a $C^2$ skew-product diffeomorphism $F:L\to L$ over $f$ so that there exists an attractor $\Lambda=\bigcap_{n\in\mathbb Z} F^n(L)$, with $F|\Lambda$ topologically conjugate to $\hat f$. See for example Appendix A in \cite{expcontendoVianaYang}. One can make the construction in our case so that $\dim L=5$. Let $\nu$ be the corresponding measure to $\hmu$, which projects to $\mu$ on $\TT$; note that $\mathrm{supp }(\nu) = \Lambda$.

Because $L$ is a contraction on the fibers $\{x\}\times D^n$, one obtains that the Lyapunov exponents of $F$ for $\nu$ are exactly $\lambda^+(f),\lambda^-(f)$ and other $n$ negative exponents. Since the Pesin formula holds for the endomorphism $f$ and $\mu$ \cite{pesinendo}, and $h_{\mu}(f)=h_{\hat\mu}(\hat f)=h_{\nu}(F)$, we obtain that the Pesin formula also holds for $F$ and $\nu$. Then classical results of Ledrappier-Young \cite{LedYoungI} tell us that the measure $\nu$ has the SRB property (absolutely continuous disintegrations along unstable manifolds). Note that when $f$ satisfies the hypothesis of \hyperlink{theoremD}{Theorem D}, so does $f^n$ for every $n \geq 1$. In particular, $\nu$ is an ergodic SRB measure for $F^n$. We are then in the hypotheses of the following theorem.

\begin{theorem}[Ledrappier] \label{thm:ledrappier}
 Let $\nu$ be an hyperbolic ergodic SRB measure for some diffeomorphism $F:L\to L$ of class $\mathcal C^2$. If $(F^n,\nu)$ is ergodic for every $n\geq 1$ then $(F,\nu)$ is Bernoulli.
 \end{theorem} 

See Theorem $5.10$ in \cite{LedrappierErg}. To conclude, if $f$ is area preserving, $C^2$, NUH, and the area $\mu$ is ergodic for all the positive iterates of $f$, then $(f,\mu)$ is Bernoulli.  

\subsection{Proofs of Theorems C and D}

We are now in a position to conclude the proofs of Theorems C and D. We start with the latter.

\begin{proof}[Proof of \hyperlink{theoremD}{Theorem D} ]\label{proofD}

	Proposition \ref{local_ergodicity} says that when $f$
	is non-uniformly hyperbolic with an almost everywhere uniform lower
	bound on the diameter of the stable manifold, then $f$ is locally ergodic.
	That is, $\mu$ is a finite or
	countable convex
	combination of measures of the form $\mu_i =  \frac{\mu \vert U_i}{\mu(U_i)}$, 
	where $U_i$ are non-empty pairwise
	disjoint open sets and each $\mu_i$  is ergodic. 
	Ergodicity of the $\mu_i$ implies that $\mu(U_i \Delta f^{-1}(U_i)) = 0$ for
	every $i$. Transitivity then implies that the decomposition is
	trivial, so $\mu$ is ergodic.
	Finally, if $f$ is non-uniformly hyperbolic with an almost everywhere uniform lower bound for the diameter of the stable manifold, the same holds for any positive iterate $f^n$ of $f$. Therefore, transitivity of $f^n$ implies ergodicity of $f^n$. Hence, by Theorem \ref{thm:ledrappier}, $f$ is Bernoulli.

\end{proof}

\begin{proof}[Proof of \hyperlink{theoremC}{Theorem C}]
The fact that $\pm 1$ are not in the spectrum of $E$ permit us to use \cite{Andersson2016} to deduce transitivity of any area preserving endomorphism isotopic to $E$. By the Proposition \ref{cor:largestable} we have a $\mathcal C^1$ open subset of $[E]\cap\mathcal U$ of endomorphisms satisfying the hypotheses of Proposition \ref{local_ergodicity}, and are thus stably ergodic by \hyperlink{theoremD}{Theorem D}.

	Note also that, since $E$ is real with $| \det E | >1 $, if $\pm 1$ is
	not in the spectrum of $E$, nor can they be in the spectrum of $E^n$
	for $n \geq 1$. Indeed, if that were the case, then $E$ would have a
	conjugate pair of non-real eigenvalues on the unit circle. But that
	would mean that $|\det E| = 1$, a contradiction.
	It follows from \cite{Andersson2016} that $f^n$ is transitive for every $n \geq 1$. Hence, by \hyperlink{theoremD}{Theorem D}, $f$ is Bernoulli.
\end{proof}

%%%%%%%%%%%%%%%%%%%%%%%%%%%%%%%%%%%%%%%%%%%%%%%%%%%%%%%%%%%%%%%%%%%%%%%%%%%%%%%%%%%%%%%%%%%%%%%%%%%%%%%%%%%%%%%%%%%%%%
\appendix
\section{Appendix: The solenoid and its invariant measures}\hypertarget{sec:appendix}{}
%%%%%%%%%%%%%%%%%%%%%%%%%%%%%%%%%%%%%%%%%%%%%%%%%%%%%%%%%%%%%%%%%%%%%%%%%%%%%%%%%%%%%%%%%%%%%%%%%%%%%%%%%%%%%%%%%%%%%%
%%%%%%%%%%%%%%%%%%%%%%%%%%%%%%%%%%%%%%%%%%%%%%%%%%%%%%%%%%%%%%%%%%%%%%%%%%%%%%%%%%%%%%%%%%%%%%%%%%%%%%%%%%%%%%%%%%%%%%

It is well-known among dynamicists that the inverse limit of an expanding circle map is homeomorphic to a so-called \emph{solenoidal attractor}, as introduced by Smale in \cite{SmaleBull} and later studied in depth by Williams \cite{Williams74}. The solenoidal attractor is an example of  a \emph{solenoidal manifold}, i.e. a topological space which is locally a product of a disk by a Cantor set. What is less known is that all this remains true for non-invertible maps in general, be them expanding or not --- as long as they are covering maps from the manifold to itself. In particular, given any surjective $C^1$ map $f: M \to M$ without critical points, its inverse limit is a solenoidal manifold. 

Although the literature on topological aspects of solenoidal manifolds is rather comprehensive (see \cite{Verjovsky22} and the references therein), an elementary and dynamic-centered exposition, including a description of invariant measures, seem to be in want of. The present work relies quite heavily on the solenoidal nature of $\hTTf$ and we think the time is ripe for a more detailed exposition. In particular, we wish to describe how $f$-invariant measures on $\TT$ ``look'' when lifted to $\hTTf$. Another thing we want to emphasize is that when $f$ and $g$ are homotopic, the spaces $\hTTf$ and $\mathbf{L}_g$ are homeomorphic; and if $f$ and $g$ are close, then $\hTTf$ and $\mathbf{L}_g$  are also close in some sense. This is particularly useful in section \ref{sec:continuity}, where we consider how measures on the projective bundles $\Pp \hTTf$ vary with the dynamics. Our exposition revolves around the construction of an abstract space ``$\Sol$'', depending on the homotopy class of $f$, to which $\hTTf$ is homeomorphic. Although our presentation deals specifically with maps on $\TT$, most of the construction can be carried over effortlessly to covering maps of other manifolds.

\subsection{The solenoid as a quotient space}

Let $E \in \Mat$ with $ d = |\det E| \geq 2$. Let $\Sigma = \{1, \ldots, d\}^{\N}$ be the set of one-sided sequences of symbols in $\{ 1, \ldots, d\}$ endowed with the (weak) product topology induced by the discrete topology on $\{1, \ldots, d\}$. Fix vectors $w_1, \ldots, w_d \in \ZZ$ such that $\{[w_1], \ldots, [w_d]\} = \ZZ / E(\ZZ)$, where $[w_i]$ denotes the element of $\ZZ / E(\ZZ)$ containing $w_i$. Upon possibly reordering, we may (and do) assume that $w_1 \in E(\ZZ)$, i.e. $[w_1] = [\underline{0}]$ is the identity element of $\ZZ / E(\ZZ)$.

We shall define a group $G$ of transformations on $\RR \times \Sigma$. In fact, $G$ is a group action of $\ZZ$ on $\RR \times \Sigma$ and, as such, it acts freely and properly discontinuously. We may thus form a space
\[ \Sol = (\RR \times \Sigma) / G,\]
called \emph{the solenoid} of $E$. It is a \emph{solenoidal 2-manifold}, meaning that it is locally a product of the Cantor set $\Sigma$ and a 2-dimensional disk. We denote the orbit of $G$ of $(\tx, \bomega) \in \RR \times \Sigma$ by $[(\tx, \bomega)]_G$ and the quotient map $(\tx, \bomega) \mapsto [(\tx, \bomega)]_G$ by $\pi_G$. As we shall see, the space $\Sol$ has the property that whenever $f: \TT \to \TT$ is a self-cover homotopic to $E$, then the natural extension $\hTTf$ of $f$ is homeomorphic to $\Sol$.  Furthermore, once we fix a lift $\tf: \RR \to \RR$ of $f$, a canonical homeomorphism $\phi: \Sol \to \hTTf$ exists for which the map $\fsol = \phi^{-1}\circ \hf \circ \phi$ has a particularly nice expression. Thus $\Sol$ can be used as a replacement for $\hTTf$ as a tool for analyzing the dynamics of $f$ itself.

\begin{remark}
In foliation theory, the previous construction is called the suspension construction. See for example Chapter $3$ in \cite{Foliated}.
\end{remark}

\subsection{Definition of \texorpdfstring{$G$}{G}}
The group $G$ consists of maps 
\[(\tx, \bomega) \mapsto (\tx+v, \psi_v(\bomega)), \quad v \in \ZZ\] where $\psi_v: \Sigma \to \Sigma$ are homeomorphisms which we now describe:

Fix some $v \in \ZZ$. Let $\bomega = (\omega_1, \omega_2, \ldots)$. Set $u_0 = v$ and, for $n \geq 1$:
\begin{itemize}
\item let $\tau_n$ be the unique number in $\{1, \ldots, d\}$ such that 
\[w_{\tau_n}-w_{\omega_n}+u_{n-1} \in E(\ZZ) \text{, and}\]
\item let $u_n$ be the unique element of $\ZZ$ such that 
\[w_{\tau_n}-w_{\omega_n}+u_{n-1} = E u_n.\]
\end{itemize} 
Finally set $\psi_v(\bomega) = (\tau_1, \tau_2, \ldots)$.

In order to understand what this group does, let us introduce some notation. Fix a lift $\tf: \RR \to \RR$ of $f$. For $i \in \{1, \ldots, d\}$, let $T_i: \RR \to \RR$ be the translation $\tx \mapsto \tx+w_i$, and write $\cF_i = \tf^{-1}\circ  T_i$. The following can be checked by straightforward induction on $n$.

\begin{proposition} \label{difference_in_Z2}
Let $\tx \in \RR$, $v \in \ZZ$, and $\bomega, \boldsymbol{\tau} \in \Sigma$. Then the following are equivalent:
\begin{enumerate}
\item $\psi_v(\bomega) = \boldsymbol{\tau}$, 
\item $\cF_{\tau_n} \circ \ldots \circ \cF_{\tau_1} (\tx+v) -\cF_{\omega_n} \circ \ldots \circ \cF_{\omega_1}(\tx) \in \ZZ $
for every $n\geq 1$.
\item $\cF_{\tau_n} \circ \ldots \circ \cF_{\tau_1} (\tx+v) -\cF_{\omega_n} \circ \ldots \circ \cF_{\omega_1}(\tx) = u_n $
for every $n\geq 1$.
\end{enumerate}
\end{proposition}

In other words, calculating pre-orbits of a point $x \in \TT$ can be done using a representative $\tx \in \RR$ and an appropriate sequence of the maps $\cF_1, \dots, \cF_d$. But the sequence of maps to be used depends on the representative $\tx$, and the map $\psi_v$ describes how the sequence changes when using the representative $\tx+v$ instead of $\tx$.

\subsection{A canonical homeomorphism \texorpdfstring{$\Sol \to \hTTf$}{Sol to Lf}}

  We define a map $\tilde{\phi}: \RR \times \Sigma \to \hTTf$ by setting 
\[\tilde{\phi} (\tx, \omega_1, \omega_2, \ldots) = (x_0, x_1, x_2, \ldots),\]
where 
\begin{equation} \label{x_i}
x_i = \pi_{cov} \circ  \cF_{\omega_i} \circ \ldots \circ \cF_{\omega_1} (\tx).
\end{equation}

It follows from (\ref{difference_in_Z2}) that $\tilde{\phi}$ is $G$-invariant, i.e. 
\[\tilde{\phi}(\tx, \bomega) = \tilde{\phi}(\tx+v, \psi_{v}(\bomega))\]
for every $v \in \ZZ$. Hence $\tilde{\phi}$ induces a map $\phi: \Sol \to \hTTf$ such 
\begin{equation}
\begin{tikzcd}
\RR \times \Sigma \ar[rr, "\tilde{\phi}"] \ar[rd, "\pi_G"] && \hTTf \\
& \Sol \ar[ur, "\phi"]
\end{tikzcd}
\end{equation}
commutes.

\begin{proposition} \label{phi_is_homeo}
The map  $\phi$ is a homeomorphism. 
\end{proposition}

\begin{proof}
Proposition \ref{difference_in_Z2} implies that $\tilde{\phi}(\tx, \bomega) = \tilde{\phi}(\ty, \boldsymbol{\tau})$ if and only if $(\tx, \bomega)$ and $(\ty, \boldsymbol{\tau})$ lie in the same $G$-orbit. Moreover, it is clear that $\tilde{\phi}$ is a continuous surjection. Hence (see  \cite[Corollary 1.10 and Example 1.5]{Rotman88}) it suffices to show that $\tilde{\phi}$ is an open mapping. To see why $\tilde{\phi}$ is open it suffices to check that the image under $\tilde{\phi}$ of open rectangles of the form
\begin{equation} \label{rectangle}
U \times C(\omega_1, \ldots, \omega_n)
\end{equation}
are open, as these form a basis of the topology on $\RR \times \Sigma$. Now, the image of (\ref{rectangle}) under $\tilde{\phi}$ is 
\begin{equation} \label{image_of_rectangle}
\hTTf \cap \{ (x_1, x_2, \ldots) \in (\TT)^{\N}: x_i \in U_i \text{ for }1\leq i \leq n\},
\end{equation}
where
\begin{equation} \label{U_i}
U_i = \pcov \circ \cF_{\omega_n} \circ \ldots \circ \cF_{\omega_1}(U).
\end{equation}
Since the right hand side of (\ref{U_i}) is a composition of open mappings, it follows that (\ref{image_of_rectangle}) is open in $\hTTf$. 
\end{proof}

In view of Proposition \ref{phi_is_homeo}, the natural extension $\hf : \hTTf \to \hTTf$ induces a homeomorphism $\fsol: \Sol \to \Sol$ by setting $\fsol = \phi^{-1} \circ \hf \circ \phi$. It turns out that $\fsol$ itself has a very natural expression, which we now describe.

Let $F, F^{\sharp}: \RR \times \Sigma \to \RR \times \Sigma$ be the maps 
\begin{align*}
(\tx, \omega_1, \omega_2, \ldots) & \overset{F}{\mapsto} (\tf(\tx), 1, \omega_1, \omega_2, \ldots) \\
(\tx, \omega_1, \omega_2, \ldots) & \overset{F^\sharp}{\mapsto} (\cF_{\omega_1}(\tx), \omega_2, \omega_3, \dots).
\end{align*}

(Recall that we have chosen the $w_i$ so that $[w_1]= [\underline{0}]$.) The map $F$ is not surjective and the map $F^\sharp$ is not injective, but $F^\sharp \circ F$ is the identity on $\RR \times \Sigma$.
The following proposition follows straight from the definition of $\phi$ and $\fsol$.

\begin{proposition}\label{expression_of_fsol}
The map $\fsol$ is given by
\[\fsol([(\tx, \bomega)]_G) = [F(\tx, \bomega)]_G.\] 

The inverse of $\fsol$ is given by
\[\fsol^{-1}([(\tx, \bomega)]_G) = [F^\sharp (\tx, \bomega)]_G.\]
\end{proposition}

The construction of $\Sol$ and its dynamics $\fsol$ is clarified by the following commuting diagram. 
\begin{equation}
\begin{tikzcd}
\RR \times \Sigma 
\ar[loop, out=120, in=60, distance=20, "F"]
\ar[r, "\pi_G"] \ar[d, "\proj_1", swap] 
& \Sol  \ar[r, shift left=0.4ex,  "\phi"] \ar[d, "\pi_{\Sol}", swap] \ar[loop, out=120, in=60, distance=20, "\fsol"]
& \hTTf  \ar[l, shift left=0.4ex, "\phi^{-1}"] \ar[ld, "\pext"]  \ar[loop, out=120, in=60, distance=20, "\hf"] \\
\RR \ar[r, "\pcov"] 
\ar[loop, out=210, in=150, distance=20, "\tf"]
& \TT \ar[loop, out=300, in=240, distance=20, "f"]
\end{tikzcd}
\end{equation}

\subsection{Invariant measures on \texorpdfstring{$\Sol$}{Sol}} \label{invariant_on_sol}

Generally speaking, if $(\tilde{X}, p)$ is a covering space\footnote{We do not make any connectivity assumptions neither on $X$ nor $\tilde{X}$, but we do require deck transformations to act transitively on fibers.} of $X$, then any (Borel) measure $\mu$ on $X$ induces a unique (Borel) measure $\tmu$ on $ \tilde{X}$ by setting $\tmu(A) = \mu(p(A))$ whenever $p \vert A$ is injective and extending it accordingly. We refer to $\mu$ as \emph{the lift} of $\mu$. Such a measure is necessarily invariant under any deck transformation of the covering.

Conversely, any measure $\tmu$ on $\tilde{X}$ invariant under deck transformations is the lift of a unique measure $\mu$ on $X$. We say that $\tmu$ \emph{descends} to $\mu$ whenever this is the case. Note that $\tmu$ descends to $\mu$ if and only if $\mu = p_*(\tmu \vert D)$ where $D$ is a fundamental domain of the covering space.

Now suppose we have a second covering space $(\tilde{Y}, q)$ of some space $Y$ and maps $\tilde{\pi}: \tilde{Y} \to \tilde{X}$, $\pi: Y \to X$ such that
\[
\begin{tikzcd}
\tilde{Y} \ar[r, "q"] \ar[d, "\tilde{\pi}"] & Y \ar[d, "\pi"] \\
\tilde{X} \ar[r, "p"] & X
\end{tikzcd}
\]
commutes. If $\nu$ is a measure on $Y$ which lifts to $\tilde{\nu}$ on $\tilde{Y}$, then $\tilde{\pi}_* \tilde{\nu} = \tmu$ if and only if $\pi_*\nu = \mu$. 

It is well known that, given any $f$-invariant measure, there is precisely one $\hf$-invariant measure which projects on $\TT$ by $\pext$. Now, since $\fsol$ is conjugated to $\hf$ and $\pext \circ \phi = \pi_{\Sol}$, it follows that there is precisely one $\fsol$-invariant measure $\bmu$ on $\Sol$ which projects to $\mu$ through $\pi_{\Sol}$. One of the advantages of working with $\Sol$ rather than with $\hTTf$ is that $\bmu$ has a particularly intuitive description. 

Let $\tilde{\bmu}$ be the lift of $\bmu$ to $\RR \times \Sigma$ and $\tmu$ the lift of $\mu$ on $\RR$. Since $(\pi_{\Sol})_*\bmu = \mu$ we must have that $(\proj_1)_* \tilde{\bmu} = \tmu$. In other words, 
\[\tilde{\bmu}(A \times \Sigma)  = \tmu(A)\]
for every measurable $A \subset \RR$.  Invariance of $\bmu$ under $\fsol$, together with Proposition \ref{expression_of_fsol} implies that

\begin{equation} \label{boldmutilde} 
\tilde{\bmu}(A \times C(\omega_1, \ldots, \omega_n)) = \tmu(\cF_{\omega_n} \circ \ldots \circ \cF_{\omega_1}(A)).
\end{equation}
Here $C(\omega_1, \ldots, \omega_n)$ denotes the cylinder set 
\[\{(\tau_1, \tau_2, \ldots) \in \Sigma: \tau_i = \omega_i \text{ for } 1 \leq i \leq n \}.\]

Sets of the form $A \times C(\omega_1, \ldots, \omega_n)$ form a semi-ring which generates the Borel $\sigma$-algebra on $\RR \times \Sigma$. Hence (\ref{boldmutilde}) determines $\tilde{\bmu}$. 

\begin{remark}
If we did not already know about the existence of $\bmu$, we could simply lift $\mu$ to $\tmu$ and use (\ref{boldmutilde}) to define $\tilde{\bmu}$ on $\RR \times \Sigma$. This measure, by construction, is $G$-invariant and descends to an $\fsol$-invariant measure $\bmu$ on $\Sol$, which in turn projects to $\mu$ through $\pi_{\Sol}$. Again, since any such measure must satisfy (\ref{boldmutilde}), it is the unique $\fsol$-invariant measure pith this property. 
\end{remark}

Below is a schematic illustration of how $\mu$, $\hmu$, $\tmu$, $\bmu$, and $\tilde{\bmu}$ are related. A filled arrow ($\rightarrow$) indicates that measures are related by push-forward, whereas a dashed arrow ($\dashrightarrow$) indicates that the former descends to the latter.

\begin{equation} \label{relation_measures}
\begin{tikzcd}
(\RR \times \Sigma, \tilde{\bmu}) 
\ar[r, dashed, "\pi_G"] \ar[d, "\proj_1"] 
& (\Sol, \bmu)  \ar[r, shift left=0.3ex, "\phi"] \ar[d, "\pi_{\Sol}"] \ar[loop, out=120, in=60, distance=20, "\fsol"]
& (\hTTf, \hmu)  \ar[l, shift left=0.3ex, "\phi^{-1}"] \ar[ld, "\pext"]  \ar[loop, out=120, in=60, distance=20, "\hf"] \\
(\RR, \tmu) \ar[r, dashed, "\pcov"] 
& (\TT, \mu) \ar[loop, out=300, in=240, distance=20, "f"]
\end{tikzcd}
\end{equation}

\begin{remark}\label{not_invariant}
The maps $\tf: \RR \to \RR$ and $F: \RR \times \Sigma \to \RR \to \Sigma$ have been left out from (\ref{relation_measures}) as they are not measure-preserving. Indeed $\tf$ is invertible on $\RR$ and has a Jacobian strictly larger than $1$ with respect to $\tmu$. The push-forward of $\tbmu$ under $F$ is the restriction of $\tbmu$ to $\RR \times C(\omega_1)$ (the image of $F$).
\end{remark}

\subsection{Lifting the Haar measure to \texorpdfstring{$\Sol$}{Sol}}

We now turn to the particular case in which $\mu$ is the Haar measure on $\TT$. Its lift $\tmu$ on $\RR$ is the Lebesgue measure. Hence (\ref{boldmutilde}) can be written as
\begin{align}\label{tbmu}
\tilde{\bmu}(A\times C(\omega_1, \dots, \omega_n)) & = \int_{A} | \det D (\cF_{\omega_n} \circ \ldots \circ \cF_{\omega_1})(\tx)| \ d\tmu(\tx) \\
& = \int_A \tilde{\bmu}_{\tx} (C(\omega_1, \ldots \omega_n)) \ d\tmu(\tx),
\end{align}

where $\tilde{\bmu}_{\tx}$ is the unique measure on $\Sigma$ such that 
\begin{equation} \label{boldmutildextilde}
\tilde{\bmu}_{\tx} (C(\omega_1, \ldots \omega_n)) = | \det D (\cF_{\omega_n} \circ \ldots \circ \cF_{\omega_1})(\tx) |
\end{equation}
 for each cylinder $C(\omega_1, \ldots, \omega_n)$. 
 
The right hand side of (\ref{boldmutildextilde}) is a $G$-invariant function on $\RR \times \Sigma$. Hence 
\[(\psi_v)_* \tilde{\bmu}_{\tx} = \tilde{\bmu}_{\tx  + v}\]
for every $v \in \ZZ$, and $\tilde{\bmu}_{\tx}$ descends to a measure $\bmu_{x}$ on $\Sol$, where $x = \pcov(\tx)$. The push-forward $\hmu_x$ of this measure under $\phi$ is the same as the push-forward of $\tilde{\bmu}_{\tx}$ under $\tilde{\phi}$. Of course, for every $(\tx, \omega_1, \omega_2, \ldots) \in \tilde{\phi}^{-1}(x_0, x_1, \ldots)$ and every $n \geq 1$ we have
\[|\det D (\cF_{\omega_n} \circ \ldots \circ \cF_{\omega_1})(\tx)| = |\det D f^n(x_n)|^{-1},\]
so that 
\begin{equation} \label{mux}
\hmu_x(\{ (\xi_1, \xi_2, \ldots) \in \hTTf : \xi_n = x_n \}) = |\det Df^n(x_n)|^{-1}.
\end{equation}

Extending (\ref{tbmu}) to the full $\sigma$-algebra on $\RR \times \Sigma$ and descending to $\Sol$ gives
\begin{equation} \label{bmu}
\bmu(A) = \int_{\TT} \bmu_x(A) \ d\mu(x)
\end{equation}
for every measurable $A \subset \Sol$. Similarly,
\begin{equation}
\mu(A) = \int_{\TT} \hmu_x(A) \ d\mu(x)
\end{equation}
for every measurable $A \subset \hTTf$.  

\subsection{Continuity}

The measures $\hmu$, $\hmu_x$, $\bmu_x$, $\tilde{\bmu}_{\tx}$, $\bmu$ and $\tilde{\bmu}$ all depend on the map $f$. Let us stress this dependence by calling them $\hmu^f$, $\hmu_x^f$, $\bmu_x^f$, $\tilde{\bmu}_{\tx}^f$ $\bmu^f$ and $\tilde{\bmu}^f$ instead. The following is obvious from (\ref{tbmu}),  (\ref{boldmutildextilde}), (\ref{mux}) and (\ref{bmu}).

\begin{proposition}\label{prop:a4}
The measures $\hmu_x^f$, $\bmu_x^f$, $\tilde{\bmu}_{\tx}^f$ $\bmu^f$ and $\tilde{\bmu}^f$ depend continuously on $f$ in the weak* topology in their respective spaces, when seen as maps from $\operatorname{End}_\mu^1(\TT)$. Moreover, the measures $\hmu_x^f$, $\bmu_x^f$ and $\tilde{\bmu}_{\tx}$ vary continuously with $x$ in the weak* topology in their respective spaces, when seen as maps from $\TT$ and $\RR$ respectively.
\end{proposition}

If we consider the measure $\hmu^f$ as a measure on $(\TT)^{\Z_+}$, supported on $\hTTf$, rather than a measure on $\hTTf$ itself, then this measure also varies continuously with $f \in \operatorname{End}_\mu^1(\TT)$.

 \subsection{Projective bundles}

 Each of the spaces $\hTTf$, $\Sol$ and $\RR \times \Sigma$ comes with a trivial bundle
 \begin{align*}
& \Pp \hTTf  = \hTTf \times \Pp \RR \\
& \Pp \Sol   = \Sol \times \Pp \RR \\
& \Pp (\RR \times \Sigma)  = \RR \times \Sigma \times \Pp \RR.
 \end{align*}
On each of these bundles, the derivative of $f$ induces bundle maps $P\hf$, $PSf$ and $PF$, respectively, given by
\begin{align*}
 (\hx, [v]) & \overset{P\hf}{\mapsto} (\hf(\hx), [Df_{\pext(\hx)} v]) \\
 ([\tx, \bomega]_G, [v]) & \overset{PSf}{\mapsto} (Sf ([\tx, \bomega]_G), [Df_{\tx} v]) \\
 (\tx, \bomega, [v]) & \overset{PF}{\mapsto} (F(\tx, \bomega), [Df_{\tx} v])
 \end{align*}
 respectively.

 Let $i$ be the identity map on $\Pp \RR$. Then the following diagram commutes.

 \begin{equation}
\begin{tikzcd}
\Pp(\RR \times \Sigma) \ar[r, "\pi_G \times i"] \ar[d, "\proj_1"] \ar[loop, out=120, in=60, distance=20, "PF"]
& \Pp \Sol \ar[r, "\phi \times i", shift left = 0.3ex] 
\ar[d, "\proj_1"], \ar[loop, out=120, in=60, distance=20, "P\fsol"]
& \Pp \hTTf \ar[d, "\proj_1"] 
\ar[l, "\phi^{-1} \times i", shift left = 0.3ex] 
\ar[loop, out=120, in=60, distance=20, "P\hf"] \\
\RR \times \Sigma \ar[r, "\pi_G"] \ar[loop, out=300, in=240, distance=20, "F"]
& \Sol \ar[r, "\phi", shift left = 0.3ex] \ar[loop, out=300, in=240, distance=20, "\fsol"]
& \hTTf \ar[l, "\phi^{-1}", shift left = 0.3ex] \ar[loop, out=300, in=240, distance=20, "\hf"]
\end{tikzcd}
 \end{equation}

 A $PS f$-invariant measure $\bmu^p$ corresponds to a $P \hf$-invariant measure $\hmu^p = (\phi, i)_*\bmu^p$ and vice versa. Moreover, $\bmu^p$ lifts to a unique measure $\tbmu^p$ on $\Pp(\RR \times \Sigma)$ through $\pi_G \times i$, which is the same as the lift of $\hmu^p$ through $\tilde{\phi}\times i$. This measure is not invariant under $PF$ (see remark \ref{not_invariant}). However

 \begin{proposition}\label{convergence_on_projetive_bundles}
Given a measure $\bmu^{p}$ on $\Pp \Sol$ and a sequence of measures $\bmu_n^{p}$ on $\Pp \Sol$ with corresponding lifts $\tbmu_n^{p}$ and $\tbmu_n^{p}$ on $\Pp (\RR \times \Sigma)$, we have that $\bmu_n^{p}$ converges (weakly*) to $\bmu^{p}$ if and only if $\tbmu_n^{p}$ converges (weakly*) to $\tbmu^{p}$. 
 \end{proposition}

 Now, $\tbmu_n^p$ and $\tbmu^p$ are all $G \times\{i\}$-invariant. Therefore, to say that $\tbmu_n^p$ converges to $\tbmu^p$, it suffices to show that $\tbmu_n^p \vert Q$ converges to $\tbmu^p \vert Q$, where $Q \subset \Pp (\RR \times \Sigma)$ is any set containing a fundamental domain of the action $G \times \{ i \}$ and satisfying $\tbmu^p( \partial Q)=0.$ In the proof of \hyperlink{theoremB}{Theorem B} we do this by taking $Q = [0,1] \times \Sigma \times \Pp \RR$.

\addcontentsline{toc}{section}{References}	
\bibliographystyle{alpha}

\end{document}